\documentclass[graybox]{svmult}

\usepackage{mathptmx}       
\usepackage{helvet}         
\usepackage{courier}        
\usepackage{type1cm}        
\usepackage{makeidx}         
\usepackage{graphicx}        
\usepackage{multicol}        
\usepackage[bottom]{footmisc}
\usepackage{amsmath,amssymb,mathrsfs,calc,dsfont}

\makeindex             

\def\EE{\mathbb{E}}

\def\NN{\mathbb{N}}
\def\PP{\mathbb{P}}

\def\RR{\mathbb{R}}
\def\SS{\mathbb{S}}

\def\XX{\mathbb{X}}

\def\d{\delta}

\def\k{\kappa}

\def\cH{\mathcal{H}}
\def\cI{\mathcal{I}}

\def\cV{\mathcal{V}}

\def\sX{\mathscr{X}}

\def\dint{\textup{d}}

\def\I{{\mathds{1}}}

\begin{document}

\title*{Poisson point process convergence\\ and extreme values in stochastic geometry}
\titlerunning{Poisson point process convergence} 
\author{Matthias Schulte and Christoph Th\"ale}
\authorrunning{Matthias Schulte and Christoph Th\"ale}
\institute{Matthias Schulte \at Karlsruhe Institute of Technology, Institute of Stochastics, D-76128 Karlsruhe, Germany.\\ \email{matthias.schulte@kit.edu}
\and Christoph Th\"ale \at Ruhr University Bochum, Faculty of Mathematics, D-44780 Bochum, Germany.\\ \email{christoph.thaele@rub.de}}
%
%
\maketitle

\abstract{Let $\eta_t$ be a Poisson point process with intensity measure $t\mu$, $t>0$, over a Borel space $\XX$, where $\mu$ is a fixed measure. Another point process $\xi_t$ on the real line is constructed by applying a symmetric function $f$ to every $k$-tuple of distinct points of $\eta_t$. It is shown that $\xi_t$ behaves after appropriate rescaling like a Poisson point process, as $t\to\infty$, under suitable conditions on $\eta_t$ and $f$. This also implies Weibull limit theorems for related extreme values. The result is then applied to investigate problems arising in stochastic geometry, including small cells in Voronoi tessellations, random simplices generated by non-stationary hyperplane processes, triangular counts with angular constraints and non-intersecting $k$-flats. Similar results are derived if the underlying Poisson point process is replaced by a binomial point process.}

\section{Introduction}

This chapter deals with the application of the Malliavin-Chen-Stein method for Poisson approximation to problems arising in stochastic geometry. More precisely, we will develop a general framework which yields Poisson point process convergence and Weibull limit theorems for the order-statistic of a class of functionals driven by an underlying Poisson or binomial point process on an abstract state space.

To motivate our general theory, let us describe a particular situation to which our results can be applied (see Remark \ref{rem:VoronoiGraph} and also Example 4 in \cite{STScalingLimits} for more details). Let $K$ be a convex body in $\RR^d$, $d\geq 2$, (that is a compact convex set with interior points) whose volume is denoted by $\ell_d(K)$. For $t>0$ let $\eta_t$ be the restriction to $K$ of a translation-invariant Poisson point process in $\RR^d$ with intensity $t$ and let $(\theta_t)_{t>0}$ be a sequence of real numbers satisfying $t^{2/d}\theta_t\to\infty$, as $t\to\infty$. Taking $\eta_t$ as vertex set of a random graph, we connect two different points of $\eta_t$ by an edge if and only if their Euclidean distance does not exceed $\theta_t$. The so-constructed random geometric graph, or Gilbert graph, is among the most prominent random graph models (see \cite{RSTGilbert} for some recent developments and \cite{PenroseBook} for an exhaustive reference). We now consider the order statistic $\xi_t=\{M_t^{(m)}:m\in\NN\}$ defined by the edge-lengths of the random geometric graph, that is, $M_t^{(1)}$ is the length of the shortest edge, $M_t^{(2)}$ is the length of the second-shortest edge etc. Now, our general theory implies that the re-scaled point process $t^{2/d}\xi_t$ converges towards a Poisson point process on $\RR_+$ with intensity measure given by $B\mapsto \beta d\int_B u^{d-1}\,\dint u$ for Borel sets $B\subset\RR_+$, where $\beta=\kappa_d\ell_d(K)/2$ and $\kappa_d$ stands for the volume of the $d$-dimensional unit ball. Moreover, it implies that there is a constant $C>0$ only depending on $K$ such that
$$
\left|\PP(t^{2/d}M_t^{(m)}>y)-e^{-\beta y^d}\sum_{i=0}^{m-1}\frac{(\beta y^d)^i}{i!}\right|\leq C\,\max\{y^{d+1},y^{2d}\} \, t^{-2/d}
$$
for any $m\in\NN$, $y\in(0,t^{2/d}\theta_t)$ and $t\geq 1$. In particular, the distribution of the re-scaled length $t^{2/d}M_t^{(1)}$ of the shortest edge of the random graph converges, as $t\to\infty$, to a Weibull distribution with survival function $y\mapsto e^{-\beta y^d}$, $y\geq 0$, at rate $t^{-2/d}$.

Our purpose here is to establish a general framework that can be applied to a broad class of examples. We also allow the underlying point process to be a Poisson \textit{or} a binomial point process. Our main result for the Poisson case refines those in \cite{STScalingLimits} or \cite{STFlats} and improves the rate of convergence. Its proof follows the ideas of \cite{Peccati2011} and \cite{STScalingLimits}, but uses the special structure of the functional under consideration as well as recent techniques from \cite{LPS} around Mehler's formula on the Poisson space. This saves some technical computations related to the product formula for multiple stochastic integrals (cf.\ \cite{LastChapter} or \cite{LPST} and \cite{Surgailis}). In case of an underlying binomial point process we use a bound for the Poisson approximation of (classical) U-statistics from \cite{BarbourEagleson}. As application of our main results, we present a couple of examples, which continue and complement those studied in \cite{STScalingLimits} and \cite{STFlats}. These are
\begin{description}

\item[1.] cells with small (nucleus-centred) inradius in a Voronoi tessellation,
\item[2.] simplices generated by a class of rotation-invariant hyperplane processes,
\item[3.] almost collinearities and flat triangles in a planar Poisson or binomial process,
\item[4.] arbitrary length-power-proximity functionals of non-intersecting $k$-flats.
\end{description}

The rest of this chapter is organized as follows. Our main results and their framework are presented in Section \ref{sec:Results}. The application to problems arising in stochastic geometry is the content of Section \ref{sec:Examples}. The proofs of the main results are postponed to the final Section \ref{sec:Proofs}.

\section{Results}\label{sec:Results}

Let $\eta_t$ ($t>0$) be a Poisson point process on a measurable space $(\XX,{\mathscr{X}})$ with intensity measure $\mu_t:=t\mu$, where $\mu$ is a fixed $\sigma$-finite measure on $\XX$. To avoid technical complications, we shall assume in this chapter that $(\XX,\sX)$ is a standard Borel space. This ensures, for example, that any point process on $\XX$ can almost surely be represented as a sum of Dirac measures. Let further $k\in\NN$ and $f: \XX^k\to\RR$ be a measurable symmetric function. Our aim here is to investigate the point process $\xi_t$ on $\RR$ which is induced by $\eta_t$ and $f$ as follows:
\begin{equation}\label{eq:defXit}
\xi_t:=\frac{1}{k!} \sum_{(x_1,\hdots,x_k)\in\eta_{t,\neq}^k} \delta_{f(x_1,\hdots,x_k)}\,.
\end{equation}
Here $\eta_{t,\neq}^k$ stands for the set of all $k$-tuples of distinct points of $\eta_t$ and $\delta_x$ is the unit Dirac measure concentrated at the point $x\in\RR$. We shall assume that
$$
\mu_t^{k}(f^{-1}([-s,s]))<\infty\qquad{\rm for\ all}\qquad s>0\,,
$$
to ensure that $\xi_t$ is a locally finite counting measure on $\RR$.

For $m\in\NN$ we denote by $M_t^{(m)}$ the distance from the origin to the $m$-th point of $\xi_t$ on the positive half-line $\RR_+:=(0,\infty)$, and by $M_t^{(-m)}$ the distance from the origin to the $m$-th point on the negative half-line $\RR_-:=(-\infty,0]$. If $\xi_t$ has less than $m$ points on the positive or negative half-line, we put $M^{(m)}_t=\infty$ or $M_t^{(-m)}=\infty$, respectively.

Fix $\gamma\in\RR$ and for $y_1,y_2\in\RR$ define
\begin{align*}
\alpha_t(y_1,y_2) := 
& \frac{1}{k!} \int_{\XX^k}\I\{t^{-\gamma} y_1 < f(x_1,\hdots,x_k)\leq t^{-\gamma} y_2\} \, \mu_t^{k}(\dint(x_1,\hdots,x_k))\,.
\end{align*}
We remark that, as a consequence of the multivariate Mecke formula for Poisson point processes (see \cite{LastChapter}), $\alpha_t(y_1,y_2)$ can be interpreted as
$$
\alpha_t(y_1,y_2)=\frac{1}{k!} \EE \sum_{(x_1,\hdots,x_k)\in\eta_{t,\neq}^k} \I\{t^{-\gamma} y_1 < f(x_1,\hdots,x_k)\leq t^{-\gamma} y_2\}\,,
$$
which is the expected number of points of $\xi_t$ in $(t^{-\gamma}y_1,t^{-\gamma}y_2]$ if $y_1<y_2$ and zero if $y_1\geq y_2$. Moreover, let, for $k\geq 2$,
\begin{equation*}
\begin{split}
r_t(y):= \max_{1\leq \ell \leq k-1} \int_{\XX^\ell} & \bigg(\int_{\XX^{k-\ell}} \I\{|f(x_1,\hdots ,x_k)|\leq t^{-\gamma} y\}\, \mu_t^{k-\ell}(\dint(x_{\ell+1},\hdots,x_k)) \bigg)^2 \\ & \mu_t^\ell(\dint(x_1,\hdots,x_\ell))
\end{split}
\end{equation*}
for $y\geq 0$ and put $r_t\equiv 0$ if $k=1$.

\begin{theorem}\label{thm:Main}
Let $\nu$ be a $\sigma$-finite non-atomic Borel measure on $\RR$. Then, there is a constant $C\geq 1$ only depending on $k$ such that
\begin{equation*}
\begin{split}
\bigg| \PP(t^\gamma M_t^{(m)}>y)-e^{-\nu((0,y])} \sum_{i=0}^{m-1} \frac{\nu((0,y])^i}{i!}\bigg| \leq |\nu((0,y])-\alpha_t(0,y)|+C\,r_t(y)
\end{split}
\end{equation*}
and
\begin{equation*}
\begin{split}
 \bigg| \PP(t^\gamma M_t^{(-m)}\geq y)-e^{-\nu((-y,0])} \sum_{i=0}^{m-1} \frac{\nu((-y,0])^i}{i!}\bigg| \leq |\nu((-y,0])-\alpha_t(-y,0)|+C\, r_t(y)
\end{split}
\end{equation*}
for all $m\in\NN$ and $y\geq 0$.
Moreover, if
\begin{equation}\label{eqn:MainAssumptionI}
\lim\limits_{t\to\infty}\alpha_t(y_1,y_2)=\nu((y_1,y_2])\qquad\text{for\ all}\qquad y_1,y_2\in\RR \ \text{ with } \ y_1<y_2
\end{equation}
and
\begin{equation}\label{eqn:MainAssumptionII}
\lim\limits_{t\to\infty} r_t(y)=0\qquad\text{for\ all}\qquad y>0\,,
\end{equation}
the rescaled point processes $(t^\gamma \xi_t)_{t>0}$ converge in distribution to a Poisson point process on $\RR$ with intensity measure $\nu$.
\end{theorem}

\begin{remark}
Let us comment on the particular case $k=1$. Here, the point process $\xi_t$ is itself a Poisson point process on $\RR$ with intensity measure derived from $\alpha_t$ as a consequence of the famous mapping theorem, for which we refer to Section 2.3 in \cite{Kingman}. This is confirmed by our Theorem \ref{thm:Main}.
\end{remark}

\begin{remark}
Theorem \ref{thm:Main} generalizes earlier versions in \cite{STScalingLimits,STFlats}, which have a similar structure, but where the quantity
\begin{equation*}
\begin{split}
\hat{r}_t(y):= \sup_{\substack{(\hat{x}_1,\hdots,\hat{x}_\ell)\in \XX^\ell\\ 1\leq \ell \leq k-1}} \mu_t^{k-\ell} ( &\{(x_1,\hdots,x_{k-\ell})\in \XX^{k-\ell}:\\
&|f(\hat{x}_1,\hdots,\hat{x}_\ell,x_1,\hdots ,x_{k-\ell})|\leq t^{-\gamma} y\})
\end{split}
\end{equation*}
for $y\geq 0$ is considered instead of $r_t(y)$. It is easy to see that $r_t(y)$ and $\hat r_t(y)$ are related by
$$
r_t(y)\leq \inf_{\varepsilon>0}\alpha_t(-y-\varepsilon,y) \, \hat{r}_t(y) \quad \text{for all} \quad y\geq 0\,.
$$
In particular, this means that the rate of convergence of the order statistics in Theorem \ref{thm:Main} improves that in \cite{STScalingLimits,STFlats} by removing a superfluous square root from $\hat{r}_t(y)$. Moreover and in contrast to \cite{STScalingLimits,STFlats}, the constant $C$ only depends on the parameter $k$.
\end{remark}

In our applications presented in Section \ref{sec:Examples}, the function $f$ is always strictly positive so that $\xi_t$ is concentrated on $\RR_+$. Moreover, the measure $\nu$ will be of a special form. The following corollary deals with this situation. To state it, we use the convention that $\alpha_t(y):=\alpha_t(0,y)$ for $y\geq 0$.

\begin{corollary}\label{corol:Main}
Let $\beta,\tau>0$. Then there is a constant $C>0$ only depending on $k$ such that
\begin{equation*}
\begin{split}
\bigg| \PP(t^\gamma M_t^{(m)}>y)-e^{-\beta y^\tau} \sum_{i=0}^{m-1} \frac{(\beta y^\tau)^i}{i!}\bigg|\leq |\beta y^\tau-\alpha_t(y)|+C\,r_t(y)
\end{split}
\end{equation*}
for all $m\in\NN$ and $y\geq 0$. If, additionally,
\begin{equation}\label{eqn:MainAssumptionIII}
\lim_{t\to\infty} \alpha_t(y)=\beta y^\tau \quad \text{ and } \quad  \lim_{t\to\infty} r_t(y)=0\qquad\text{for\ all}\qquad y>0\,,
\end{equation}
the rescaled point processes $(t^\gamma \xi_t)_{t>0}$ converge in distribution to a Poisson point process on $\RR_+$ with the intensity measure
\begin{equation}\label{eqn:nu}
\nu(B)= \beta \tau \int_B u^{\tau-1} \, \dint u, \quad B\subset\RR_+\text{ Borel}\,.
\end{equation}
\end{corollary}

\begin{remark}
The limiting Poisson point process appearing in the context of Corollary \ref{corol:Main} is usually called a Weibull process on $\RR_+$, the reason for this being that the distance from the origin to the next point follows a Weibull distribution.
\end{remark}

If $\mu$ is a finite measure, i.e., if $\mu(\XX)<\infty$, one can replace the underlying Poisson point process $\eta_t$ by a binomial point process $\zeta_n$ having a fixed number of $n$ points which are independent and identically distributed according to the probability measure $\mu(\,\cdot\,)/\mu(\XX)$. Without loss of generality we assume that $\mu(\XX)=1$ in what follows. In this situation, we consider instead of $\xi_t$ defined at \eqref{eq:defXit} the derived point process $\widehat{\xi}_n$ on $\RR$ given by
$$
\widehat{\xi}_n:=\frac{1}{k!}\sum_{(x_1,\hdots,x_k)\in\zeta^k_{n,\neq}} \delta_{f(x_1,\hdots,x_k)}\,,
$$
where $\zeta^k_{n,\neq}$ stands for the collection of all $k$-tuples of distinct points of $\zeta_n$. For $m\in\NN$ let $\widehat{M}_n^{(m)}$ and $\widehat{M}_n^{(-m)}$ be defined similarly as $M_n^{(m)}$ and $M_n^{(-m)}$ above with $\xi_t$ replaced by $\widehat{\xi}_n$. For $n,k\in\NN$ we denote by $(n)_k$ the descending factorial $n\cdot(n-1)\cdot\hdots\cdot(n-k+1)$. Using the notation
\begin{eqnarray*}
\alpha_n(y_1,y_2) &:=& {(n)_k\over k!}\int_{\XX^k}\I\{n^{-\gamma}y_1<f(x_1,\ldots,x_k)\leq n^{-\gamma}y_2\}\,\mu^k(\dint(x_1,\ldots,x_k))\,,\\
r_n(y) &:=& \max_{1\leq\ell\leq k-1}(n)_{2k-\ell}\int_{\XX^\ell}\bigg(\int_{\XX^{k-\ell}}\I\{|f(x_1,\ldots,x_k)|\leq n^{-\gamma}y\}\\
&&\qquad\qquad\qquad\qquad\qquad \mu^{k-\ell}(\dint(x_{\ell+1},\ldots,x_k))\bigg)^2\,\mu^{\ell}(\dint(x_1,\ldots,x_\ell))
\end{eqnarray*}
for $y_1,y_2,y\in\RR$, we can now present the binomial counterpart of Theorem \ref{thm:Main}.

\begin{theorem}\label{thm:MainBinomial}
Let $\mu$ be a probability measure on $\XX$ and $\nu$ be a $\sigma$-finite non-atomic Borel measure on $\RR$. Then, there is a constant $C\geq 1$ only depending on $k$ such that
\begin{equation*}
\begin{split}
\bigg| \PP(n^\gamma \widehat{M}_n^{(m)}>y)-e^{-\nu((0,y])} \sum_{i=0}^{m-1} \frac{\nu((0,y])^i}{i!}\bigg| &\leq |\nu((0,y])-\alpha_n(0,y)|\\
&\qquad+C\Big(r_n(y)+{\alpha_n(0,y)\over n}\Big)
\end{split}
\end{equation*}
and
\begin{equation*}
\begin{split}
 \bigg| \PP(n^\gamma \widehat{M}_n^{(-m)}\geq y)-e^{-\nu((-y,0])} \sum_{i=0}^{m-1} \frac{\nu((-y,0])^i}{i!}\bigg| &\leq |\nu((-y,0])-\alpha_n(-y,0)|\\
 &\qquad+C\Big(r_n(y)+{\alpha_n(-y,0)\over n}\Big)
\end{split}
\end{equation*}
for all $m\in\NN$ and $y\geq 0$.
Moreover, if
\begin{equation*}\label{eqn:MainAssumptionI-Binom}
\lim\limits_{n\to\infty}\alpha_n(y_1,y_2)=\nu((y_1,y_2])\qquad\text{for\ all}\qquad y_1,y_2\in\RR \ \text{ with } \ y_1<y_2
\end{equation*}
and
\begin{equation*}\label{eqn:MainAssumptionII-Binom}
\lim\limits_{n\to\infty} r_n(y)=0\qquad\text{for\ all}\qquad y>0\,,
\end{equation*}
the rescaled point processes $(n^\gamma \widehat{\xi}_n)_{n\geq 1}$ converge in distribution to a Poisson point process on $\RR$ with intensity measure $\nu$.
\end{theorem}


As in the Poisson case, Theorem \ref{thm:MainBinomial} allows a re-formulation as in Corollary \ref{corol:Main} for the special situation in which $f$ is non-negative and $\nu$ has a power-law density. As above, we use the convention that $\alpha_n(y):=\alpha_n(0,y)$ for $y\geq 0$.

\begin{corollary}\label{corol:MainBinomial}
Let $\beta,\tau>0$. Then there is a constant $C>0$ only depending on $k$ such that
\begin{equation*}
\begin{split}
\bigg| \PP(n^\gamma\widehat{M}_n^{(m)}>y)-e^{-\beta y^\tau} \sum_{i=0}^{m-1} \frac{(\beta y^\tau)^i}{i!}\bigg|\leq |\beta y^\tau-\alpha_n(y)|+C\Big(r_n(y)+{\alpha_n(y)\over n}\Big)
\end{split}
\end{equation*}
for all $m\in\NN$ and $y\geq 0$. If, additionally,
\begin{equation*}
\lim_{n\to\infty} \alpha_n(y)=\beta y^\tau \quad \text{ and } \quad  \lim_{n\to\infty} r_n(y)=0\qquad\text{for\ all}\qquad y>0\,,
\end{equation*}
the rescaled point processes $(n^\gamma\widehat{\xi}_n)_{n\geq 1}$ converge in distribution to a Poisson point process on $\RR_+$ with intensity measure given by \eqref{eqn:nu}.
\end{corollary}

\section{Examples}\label{sec:Examples}

In this section we apply the results presented above to problems arising in stochastic geometry. The minimal nucleus-centred inradius of the cells of a Voronoi tessellation is considered in Section \ref{subsec:Voronoi}. This example is inspired by the work \cite{CalkaChenavier} and was not previously considered in \cite{STScalingLimits}, although it is closely related to the minimal edge length of the random geometric graph discussed in the introduction. Our next example generalizes Example 6 of \cite{STScalingLimits} from the translation-invariant case to arbitrary distance parameters $r\geq 1$. In dimension two it also sheds some new light onto the area of small cells in line tessellations. Our third example is inspired by a result in \cite{SilvermanBrown} and deals with approximate collinearities and flat triangles induced by a planar Poisson or binomial point process. 
Our last example deals with non-intersecting $k$-flats. The result generalizes Example 1 in \cite{STScalingLimits} and one of the results in \cite{STFlats} to arbitrary distance powers $a>0$.

\subsection{Voronoi tessellations}\label{subsec:Voronoi}

For a finite set $\chi\neq\emptyset$ of points in $\RR^d$, $d\geq 2$, the Voronoi cell $v(x,\chi)$ with nucleus $x\in\chi$ is the (possibly unbounded) set $$v(x,\chi)=\big\{z\in\RR^d:\|x-z\|\leq\|x'-z\|\ \text{for all}\ x'\in\chi\setminus\{x\}\big\}$$ of all points in $\RR^d$ having $x$ as their nearest neighbour in $\chi$. 
The family $$\cV_\chi=\{v(x,\chi):x\in\chi\}$$ subdivides $\RR^d$ into a finite number of random polyhedra, which form the so-called Voronoi tessellation associated with $\chi$, see \cite[Chapter 10.2]{SW}. For $\chi=\emptyset$ we put $\cV_\emptyset=\{\RR^d\}$. One characteristic measuring the size of a Voronoi cell $v(x,\chi)$ is its nucleus-centred inradius $R(x,\chi)$. It is defined as the radius of the largest ball included in $v(x,\chi)$ and having $x$ as its midpoint. Note that $R(x,\chi)$ takes the value $\infty$ if $\chi=\{x\}$. Define
$$
R(\cV_\chi):=\min\{R(x,\chi):x\in\chi\}
$$
for non-empty $\chi$ and $R(\cV_\emptyset):=\infty$.

In \cite{CalkaChenavier} the asymptotic behaviour of $R(\cV_\chi)$ has been investigated in the case that $\chi$ is a Poisson point process in a convex body $K$ of intensity $t>0$, as $t\to\infty$. Using Corollary \ref{corol:Main} we can get back one of the main results of \cite{CalkaChenavier} and add a rate of convergence to the limit theorem (compare with \cite[Equation (2b)]{CalkaChenavier} in particular). Moreover, we provide a similar result for an underlying binomial point process.

\begin{corollary}\label{corol:Voronoi}
Let $\eta_t$ be a Poisson point process with intensity measure $t\ell_d|_K$, where $\ell_d|_K$ stands for the restriction of the Lebesgue measure to a convex body $K$ and $t>0$. Then, there exists a constant $C>0$ depending on $K$ such that
$$
\left|\PP\big(t^{2/d} R(\cV_{\eta_t})>y\big)-e^{-2^{d-1}\kappa_d\ell_d(K)y^d}\right|\leq C\,t^{-2/d} \max\{y^{d+1},y^{2d}\}
$$
for all $y\geq 0$ and $t\geq 1$. In addition, if $\zeta_n$ is a binomial point process with $n\geq 2$ independent points distributed according to $\ell_d(K)^{-1}\,\ell_d|_K$, then
$$
\left|\PP\big(n^{2/d} R(\cV_{\zeta_n})>y\big)-e^{-2^{d-1}\kappa_d\ell_d(K)y^d}\right|\leq C\,n^{-2/d} \max\{y^d,y^{2d}\}
$$
for $y\geq 0$ and with a constant $C>0$ depending on $K$.
\end{corollary}
\begin{proof}
To apply Corollary \ref{corol:Main} we first have to investigate $\alpha_t(y)$ for fixed $y>0$. For this we abbreviate $\cV_{\eta_t}$ by $\cV_t$ and observe that -- by definition of a Voronoi cell -- $R(\cV_t)$ is half of the minimal interpoint distance of points from $\eta_t$, i.e.
$$
R(\cV_t)=\frac{1}{2}\,\min\big\{\|x_1-x_2\|:(x_1,x_2)\in\eta_{t,\neq}^2\}\,.
$$
Consequently, we have
\begin{eqnarray*}
\alpha_t(y) 
&=& \frac{t^2}{2}\int_K\int_K\I\{\|x_1-x_2\|\leq 2yt^{-\gamma}\}\,\dint x_2 \, \dint x_1\\
&=& \frac{t^2}{2}\int_{\RR^d}\ell_d(K\cap B_{2yt^{-\gamma}}^d(x_1))\,\dint x_1- \frac{t^2}{2}\int_{\RR^d\setminus K}\ell_d(K\cap B_{2yt^{-\gamma}}^d(x_1))\,\dint x_1\,,
\end{eqnarray*}
where $B_r^d(x)$ is the $d$-dimensional ball of radius $r>0$ around $x\in\RR^d$. From Theorem 5.2.1 in \cite{SW} (see Equation (5.14) in particular) it follows that
$$
\frac{t^2}{2}\int_{\RR^d}V_d(K\cap B_{2yt^{-\gamma}}^d(x_1))\,\dint x_1=\frac{t^2}{2}\,\ell_d(K)\,\kappa_d(2yt^{-\gamma})^d=2^{d-1}\ell_d(K)\kappa_dy^dt^{2-\gamma d}\,.
$$
Moreover, Steiner's formula \cite[Equation (14.5)]{SW} yields
\begin{eqnarray*}
&&\frac{t^2}{2}\int_{\RR^d\setminus K}\ell_d(K\cap B_{2yt^{-\gamma}}^d(x_1))\,\dint x_1\\
&&\qquad \leq \frac{\kappa_d}{2}\,t^2(2yt^{-\gamma})^d\,\ell_d(\{z\in\RR^d\setminus K:\inf_{z'\in K}\|z-z'\|\leq 2yt^{-\gamma}\})\\
&&\qquad = \frac{\kappa_d}{2}t^2(2yt^{-\gamma})^d\,\sum_{j=0}^{d-1}\kappa_{d-j}V_j(K)(2yt^{-\gamma})^{d-j}\,,
\end{eqnarray*}
where $V_0(K),\ldots,V_{d-1}(K)$ are the so-called intrinsic volumes of $K$, cf.\ \cite{SW}. Choosing $\gamma=2/d$, this implies that $\alpha_t(y)$ is dominated by its first integral term and that
$$
\big|\alpha_t(y)-2^{d-1}\kappa_d \ell_d(K) y^d\big|\leq c_1\,t^{-2/d} \max\{y^{d+1},y^{2d}\}
$$
for $t\geq 1$ with a constant $c_1$ only depending on $K$.

Finally, we have to deal with $r_t(y)$. Here, we have
\begin{align*}
r_t(y) &= t^3 \int_K\bigg(\int_K\I\{\|x-y\|\leq 2yt^{-\gamma}\} \, \dint y\bigg)^2 \,\dint x\\
& \leq  t^{3}\ell_d(K) \, (t^{-2}\kappa_d 2^dy^{d})^2 = \ell_d(K) \, 4^d \kappa_d^2 y^{2d} t^{-1}\,.
\end{align*}
In the binomial case, one can derive analogous bounds for $\alpha_n(y)$ and $r_n(y)$, $y>0$.
Since $\min(2/d,1)=2/d$ for all $d\geq 2$, application of Corollary \ref{corol:Main} and Corollary \ref{corol:MainBinomial} completes the proof.\hfill $\Box$
\end{proof}

\begin{remark}\label{rem:VoronoiGraph}
We have used in the proof that $R(\cV_{\eta_t})$ is half of the minimal inter-point distance between points of $\eta_t$ in $K$. Thus, Corollary \ref{corol:Voronoi} also makes a statement about this minimal inter-point distance. Consequently, $2R(\cV_{\eta_t})$ is also the same as the shortest edge length of a random geometric graph based on $\eta_t$ as discussed in the introduction (cf.\ \cite{RSTGilbert} and \cite{PenroseBook} for an exhaustive reference on random geometric graphs) or as the shortest edge length of a Delaunay graph (see \cite{SKM,SW} for background material on Delaunay graphs or tessellations). A similar comment applies if $\eta_t$ is replaced by a binomial point process $\zeta_n$.
\end{remark}

\subsection{Hyperplane tessellations}\label{subsec:Hyperplanes}

Let $\cH$ be the space of hyperplanes in $\RR^d$, fix a distance parameter $r\geq 1$ and a convex body $K\subset\RR^d$, and define a (finite) measure $\mu$ on $\cH$ by the relation
$$
\int_{\cH}g(H)\,\mu(\dint H) = \int_{\SS^{d-1}}\int_{0}^\infty g(u^\perp+p u)\,\I\{(u^\perp+p u)\cap K\neq\emptyset\} \, p^{r-1}\,\dint p\,\dint u\,,
$$
where $g\geq 0$ is a measurable function on $\cH$, $u^\perp$ is the linear subspace of all vectors that are orthogonal to $u$, and $\dint u$ stands for the infinitesimal element of the normalized Lebesgue measure on the $(d-1)$-dimensional unit sphere $\SS^{d-1}$. By $\eta_t$ we mean in this section a Poisson point process on $\cH$ with intensity measure $\mu_t:=t\mu$, $t>0$. Let us further write for $n\in\NN$ with $n\geq d+1$, $\zeta_n$ for a binomial process on $\cH$ consisting of $n\in\NN$ hyperplanes distributed according to the probability measure $\mu(\cH)^{-1}\,\mu$.

If $K=\RR^d$ in the Poisson case, one obtains a tessellation of the whole $\RR^d$ into bounded cells. In this context one is interested in the so-called zero cell $Z_0$, which is the almost surely uniquely determined cell containing the origin. If $r=1$, $Z_0$ has the same distribution as the zero-cell of a rotation- {and} translation-invariant Poisson hyperplane tessellation. If $r=d$, $Z_0$ is equal in distribution to the so-called typical cell of a Poisson-Voronoi tessellation as considered in the previous section, see \cite{SW}. Thus, the tessellation induced by $\eta_t$ interpolates in some sense between the translation-invariant Poisson hyperplane and the Poisson-Voronoi tessellation, which explains the recent interest in this model \cite{HugHoermann,HHRT,HugSchneider}. For more background material about random tessellations (and in particular Poisson hyperplane and Poisson-Voronoi tessellations) we refer to Chapter 10 in \cite{SW} and Chapter 9 in \cite{SKM}.

We are interested here in the simplices generated by the hyperplanes of $\eta_t$ or $\zeta_n$, which are contained in the prescribed convex set $K$. For a $(d+1)$-tuple $(H_1,\ldots,H_{d+1})$ of distinct hyperplanes of $\eta_t$ or $\zeta_n$ let us write $[H_1,\ldots,H_{d+1}]$ for the simplex generated by $H_1,\ldots,H_{d+1}$ and define the point processes
$$\xi_t:={1\over(d+1)!}\sum_{(H_1,\ldots,H_{d+1})\in\eta_{t,\neq}^{d+1}}\delta_{\ell_d([H_1,\ldots,H_{d+1}])}\,\I\{[H_1,\ldots,H_{d+1}]\subset K\}$$
and
$$\widehat{\xi}_n:={1\over(d+1)!}\sum_{(H_1,\ldots,H_{d+1})\in\zeta_{n,\neq}^{d+1}}\delta_{\ell_d([H_1,\ldots,H_{d+1}])}\,\I\{[H_1,\ldots,H_{d+1}]\subset K\}\,.$$
By $M_t^{(m)}$ and $\widehat{M}_n^{(m)}$ we mean the $m$th order statistics associated with $\xi_t$ and $\widehat{\xi}_n$, respectively. In particular $M_t^{(1)}$ and $\widehat{M}_n^{(1)}$ are the smallest volume of a simplex included in $K$. Moreover, for fixed hyperplanes $H_1,\ldots,H_d$ in general position let $z(H_1,\hdots,H_d):=H_1\cap\ldots\cap H_d$ be the intersection point of $H_1,\hdots,H_d$. By $H_{\delta,u}$ we denote the hyperplane with unit normal vector $u\in\SS^{d-1}$ and distance $\delta>0$ to the origin. The following result generalizes \cite[Theorem 2.6]{STScalingLimits} from the translation-invariant case $r=1$ to arbitrary distance parameter $r\geq 1$.

\begin{corollary}\label{Corol:Hyperplanes}
Define
\begin{equation*}
\begin{split}
\beta:=\frac{1}{(d+1)!}&\int_{\cH^d}\int_{\SS^{d-1}}\I\{H_1\cap\ldots\cap H_d\cap K\neq\emptyset\}\,|u^Tz(H_1,\ldots,H_d)|^{r-1}\\
&\times \ell_d([H_1,\ldots,H_d,z(H_1,\hdots,H_d)+H_{1,u}])^{-1/d} \, \dint u\,\mu^d(\dint (H_1,\ldots,H_d))\,.
\end{split}
\end{equation*}
Then $t^{{d}(d+1)}\xi_t$ and $n^{{d}(d+1)}\widehat{\xi}_n$ converge, as $t\to\infty$ or $n\to\infty$, in distribution to a Poisson point process on $\RR_+$ with intensity measure given by $$B\mapsto \frac{\beta}{d}\int_B u^{(1-d)/d}\,\dint u$$ for Borel sets $B\subset\RR_+$. In particular, for each $m\in\NN$, $t^{{d}(d+1)}M_t^{(m)}$ and $n^{{d}(d+1)}\widehat{M}_n^{(m)}$ converge towards a random variable with survival function
$$
y\mapsto\exp\big(-{\beta}\,y^{1/d}\big)\sum_{i=0}^{m-1}\frac{(\beta y^{1/d})^i}{i!}, \quad y\geq 0\,\,.
$$
\end{corollary}

\begin{proof}
For $y>0$ we have
\begin{equation*}
\begin{split}
\alpha_t(y) = \frac{t^{d+1}}{(d+1)!}\int_{\cH^{d+1}} & \I\{[H_1,\ldots,H_{d+1}]\subset K\}\\
&\times\I\{\ell_d([H_1,\ldots,H_{d+1}])\leq yt^{-\gamma}\}\,\mu^{d+1}(\dint(H_1,\ldots,H_{d+1}))\,.
\end{split}
\end{equation*}
For fixed hyperplanes $H_1,\ldots,H_d$ in general position we parametrize $H_{d+1}$ by a pair $(\delta,u)\in[0,\infty)\times\SS^{d-1}$, where $\delta$ is the distance of $H_{d+1}$ to the origin. Then $\alpha_t(y)$ can be re-written as
\begin{equation}\label{eq:alphaHyperplanes}
\begin{split}
\alpha_t(y) &= \frac{1}{2(d+1)!}\int_{\cH^d}\int_{\SS^{d-1}}\int_{-\infty}^\infty t^{d+1}\I\{[H_1,\ldots,H_d,H_{\delta,u}]\subset K\}\\
& \quad \times\I\{\ell_d([H_1,\ldots,H_d,H_{\delta,u}])\leq yt^{-\gamma}\} |\delta|^{r-1}\,\dint\delta \, \dint u \, \mu^d(\dint(H_1,\ldots,H_d))\,.
\end{split}
\end{equation}
Since the hyperplane $H_{\delta,u}$ has the distance $|u^T z(H_1,\hdots,H_d)-\delta|$ to $z(H_1,\hdots,H_d)$, we have that
\begin{align*}
&\ell_d([H_1,\ldots,H_d,H_{\delta,u}])\\
&\qquad=|u^T z(H_1,\hdots,H_d)-\delta|^d\,\ell_d([H_1,\ldots,H_d,z(H_1,\hdots,H_d)+H_{1,u}])\,.
\end{align*}
Let $\gamma=d(d+1)$ and $M:=\max\{\|z\|^{r-1}:z\in K\}$. For fixed $H_1,\hdots,H_d\in \cH^d$ such that $H_1\cap\hdots\cap H_d\cap K\neq \emptyset$ and $u\in\SS^{d-1}$ we can estimate the inner integral in \eqref{eq:alphaHyperplanes} from above by
\begin{equation*}
\begin{split}
& M\int_{-\infty}^\infty t^{d+1} \I\{|u^T z(H_1,\hdots,H_d)-\delta|^d\\
& \hskip 3cm \ell_d([H_1,\ldots,H_d,z(H_1,\hdots,H_d)+H_{1,u}])\leq yt^{-\gamma}\}\,\dint\delta \\
& \leq 2 M\, \ell_d([H_1,\ldots,H_d,z(H_1,\hdots,H_d)+H_{1,u}])^{-1/d}\, y^{1/d}.
\end{split}
\end{equation*}
The hyperplanes $H_1-z(H_1,\hdots,H_d),\ldots,H_d-z(H_1,\hdots,H_d)$ partition the unit sphere $\SS^{d-1}$ into $2^d$ spherical caps $S_1,\ldots,S_{2^d}$. For each $u\in S_j$ ($1\leq j\leq 2^d$), transformation into spherical coordinates shows that
$$
\ell_d([H_1,\ldots,H_d,z(H_1,\ldots,H_d)+H_{1,u}])\geq c_d\,\ell_{d-1}(S_j)\,,
$$
where $c_d>0$ is a dimension dependent constant and $\ell_{d-1}(S_j)$ is the spherical Lebesgue measure of $S_j$. Consequently, we have
\begin{equation*}
\begin{split}
\alpha_t(y)&\leq \frac{M}{(d+1)!}\int_{\cH^d} \I\{H_1\cap\ldots\cap H_d\cap K\neq\emptyset\}\\
&\qquad\qquad\qquad\times\sum_{j=1}^{2^d}\int_{S_j}\left(\frac{y}{c_d\,\ell_{d-1}(S_j)}\right)^{1/d}\,\dint u \, \mu^d(\dint(H_1,\ldots,H_d))\\
&\leq \frac{M}{(d+1)!}\int_{\cH^d} \I\{H_1\cap\ldots\cap H_d\cap K\neq\emptyset\}\\
&\qquad\qquad\qquad\times\sum_{j=1}^{2^d}\ell_{d-1}(S_j)\left(\frac{y}{c_d\,\ell_{d-1}(S_j)}\right)^{1/d}\,\mu^d(\dint(H_1,\ldots,H_d))\,.
\end{split}
\end{equation*}
Since the last expression is finite, we can apply the dominated convergence theorem in \eqref{eq:alphaHyperplanes}. By the same arguments we used to obtain an upper bound for the inner integral in \eqref{eq:alphaHyperplanes}, we see that, for $H_1,\ldots, H_d\in\cH^d$ and $u\in\SS^{d-1}$,
\begin{equation*}
\begin{split}
& \lim\limits_{t\to\infty}\int_{-\infty}^\infty t^{d+1}\I\{[H_1,\ldots,H_d,H_{\delta,u}]\subset K\} \I\{\ell_d([H_1,\ldots,H_d,H_{\delta,u}])\leq yt^{-\gamma}\} |\delta|^{r-1}\,\dint\delta\\
& =  2 \I\{H_1\cap\ldots\cap H_d\cap K\neq\emptyset\}\, \ell_d([H_1,\ldots,H_d,z(H_1,\hdots,H_d)+H_{1,u}])^{-1/d}\\
& \qquad\qquad\qquad\times  |u^T z(H_1,\hdots,H_d)|^{r-1} y^{1/d}\,.
\end{split}
\end{equation*}
Altogether, we obtain that
$$
\lim\limits_{t\to\infty} \alpha_t(y)=\beta y^{1/d}\,.
$$
By the same estimates as above, we have that, for any $\ell\in\{1,\hdots,d\}$,
\begin{align*}
& t^{\ell} \int_{\cH^\ell} \bigg(t^{d+1-\ell}\int_{\cH^{d+1-\ell}} \I\{ [H_1,\hdots,H_{d+1}]\subset H, V_d([H_1,\hdots,H_{d+1}])\leq yt^{-\gamma} \}\\
& \hskip 3cm \mu^{d+1-\ell}(\dint(H_{\ell+1},\hdots,H_{d+1})) \bigg)^2\mu^\ell(\dint(H_1,\hdots,H_\ell))\\
& \leq t^{\ell} \int_{\cH^\ell} \bigg(M t^{-\ell}\int_{\cH^{d-\ell}} \int_{\SS^{d-1}} \I\{H_1\cap\hdots\cap H_d\cap K\neq \emptyset\} \, y^{1/d}\\
& \quad  \ell_d([H_1,\hdots,H_d,z(H_1,\hdots,H_d)+H_{1,u}])^{-1/d} \, \dint u \,\mu^{d-\ell}(\dint(H_{\ell+1},\hdots,H_{d})) \bigg)^2\\
& \hspace{7cm}\mu^\ell(\dint(H_1,\hdots,H_\ell))\,.
\end{align*}
Hence, $r_t(y)\to 0$ as $t\to\infty$ so that application of Corollary \ref{corol:Main} completes the proof of the Poisson case. The result for an underlying binomial point process follows from similar estimates and Corollary \ref{corol:MainBinomial}.\hfill $\Box$
\end{proof}

\begin{remark}
Although Corollary \ref{corol:Main} or Corollary \ref{corol:MainBinomial} deliver a rate of convergence, we cannot provide such rate for this particular example. This is due to the fact that the exact asymptotic behaviour of $\alpha_t(y)$ or $\alpha_n(y)$ depends in a delicate way on the smoothness behaviour of the boundary of $K$.
\end{remark}

Corollary \ref{Corol:Hyperplanes} admits a nice interpretation in the planar case $d=2$. Namely, the smallest triangle contained in $K$ coincides with the smallest triangular cell included in $K$ of the line tessellation induced by $\eta_t$ or $\zeta_n$ (note that this argument fails in higher dimensions). This way, Corollary \ref{Corol:Hyperplanes} also makes a statement about the area of small triangular cells, which generalizes Corollary 2.7 in \cite{STScalingLimits} from the translation-invariant case $r=1$ to arbitrary distance parameters $r\geq 1$:

\begin{corollary}
Denote by $A_t$ or $A_n$ the area of the smallest triangular cell in $K$ of a line tessellation generated by a Poisson line process $\eta_t$ or a binomial line process $\zeta_n$ with distance parameter $r\geq 1$, respectively. Then $t^{6}A_t$ and $n^{6}A_n$ both converge in distribution, as $t\to\infty$ or $n\to\infty$, to a Weibull random variable with survival function $y\mapsto\exp(-{\beta}\,y^{1/2})$, $y\geq 0$, where $\beta$ is as in Corollary \ref{Corol:Hyperplanes}.
\end{corollary}

\subsection{Flat triangles}\label{subsec:FlatTriangles}

So-called \textit{ley lines} are expected alignments of a set of locations that are of geographical and/or historical interest, such as ancient monuments, megaliths and natural ridge-tops \cite{LeyHunting}. For this reason, there is some interest in archaeology, for example, to test a point pattern on spatial randomness against an alternative favouring collinearities. We carry out this program in case of a planar Poisson or binomial point process and follow \cite[Section 5]{SilvermanBrown}, where the asymptotic behaviour of the number of so-called flat triangles in a binomial point process has been investigated. 

Let $K$ be a convex body in the plane and let $\mu$ be a probability measure on $K$ which has a continuous density $\varphi$ with respect to the Lebesgue measure $\ell_2|_K$ restricted to $K$. By $\eta_t$ we denote a Poisson point process with intensity measure $\mu_t:=t\mu$, $t>0$, and by $\zeta_n$ a binomial process of $n\geq 1$ points which are independent and identically distributed according to $\mu$. For a triple $(x_1,x_2,x_3)$ of distinct points of $\eta_t$ or $\zeta_n$ we let $\theta(x_1,x_2,x_3)$ be the largest angle of the triangle formed by $x_1,x_2$ and $x_3$.
We can now build the point processes
$$\xi_t:={1\over 6}\sum_{(x_1,x_2,x_3)\in\eta_{t,\neq}^3}\delta_{\pi-\theta(x_1,x_2,x_3)}$$
and
$$\widehat{\xi}_n:={1\over 6}\sum_{(x_1,x_2,x_3)\in\zeta_{n,\neq}^3}\delta_{\pi-\theta(x_1,x_2,x_3)}$$
on the positive real half-line. The interpretation is as follows: if for a triple  $(x_1,x_2,x_3)$ in $\eta_{t,\neq}^3$ or $\zeta_{n,\neq}^3$ the value $\pi-\theta(x_1,x_2,x_3)$ is small, then the triangle formed by these points is flat in the sense that its height on the longest side is small.

\begin{corollary}\label{cor:flattriangles}
Define
\begin{equation*}
\beta:=\int_K\int_K\int_0^1 s(1-s)\,\varphi(sx_1+(1-s)x_2)\,\|x_1-x_2\|^2\,\dint s\,\mu(\dint x_1) \, \mu(\dint x_2).
\end{equation*}
Further assume that the density $\varphi$ is Lipschitz continuous. Then the re-scaled point processes $t^3\xi_t$ and $n^3\widehat\xi_n$ both converge in distribution to a homogeneous Poisson point process on $\RR_+$ with intensity $\beta$, as $t\to\infty$ or $n\to\infty$, respectively. In addition, there is a constant $C_y>0$ depending on $K$, $\varphi$ and $y$ such that
$$
\left|\PP(t^3M_t^{(m)}>y)-e^{-\beta y}\sum_{i=0}^{m-1}\frac{(\beta y)^i}{i!}\right|\leq C_y\,t^{-1}
$$
and
$$
\left|\PP(n^3M_n^{(m)}>y)-e^{-\beta y}\sum_{i=0}^{m-1}\frac{(\beta y)^i}{i!}\right|\leq C_y\,n^{-1}
$$
for all $t\geq 1$, $n\geq 3$ and $m\in\NN$.
\end{corollary}

\begin{proof}
To apply Corollary \ref{corol:Main} we have to consider the limit behaviour of $\alpha_t(y)$ and $r_t(y)$ for fixed $y>0$, as $t\to\infty$. For $x_1,x_2\in K$ and $\varepsilon>0$ define $A(x_1,x_2,\varepsilon)$ as the set of all $x_3\in K$ such that $\pi-\theta(x_1,x_2,x_3)\leq \varepsilon$. Then we have
$$
\alpha_t(y) = \frac{t^3}{6}\int_K\int_K\int_K\I\{x_3\in A(x_1,x_2,yt^{-\gamma})\}\,\varphi(x_1)\varphi(x_2)\varphi(x_3)\,\dint x_3 \, \dint x_2 \, \dint x_1\,.
$$
Without loss of generality we can assume that $x_3$ is the vertex adjacent to the largest angle. We indicate this by writing $x_3={\rm LA}(x_1,x_2,x_3)$. We parametrize $x_3$ by its distance $h$ to the line through $x_1$ and $x_2$ and the projection of $x_3$ onto that line, which can be represented as $sx_1+(1-s)x_2$ for some $s\in[0,1]$. Writing $x_3=x_3(s,h)$, we obtain that
\begin{align*}
\alpha_t(y) = \frac{t^3}{2}\int_K\int_K\int_0^1\int_{-\infty}^\infty & \I\{x_3(s,h)\in A(x_1,x_2,yt^{-\gamma}),x_3={\rm LA}(x_1,x_2,x_3)\}\, \\
&\times\varphi(x_1) \varphi(x_2)\varphi(x_3(s,h)) \|x_1-x_2\| \,\dint h \, \dint s \, \dint x_2 \, \dint x_1\,.
\end{align*}
The sum of the angles at $x_1$ and $x_2$ is given by
$$
\arctan(|h|/(s\|x_1-x_2\|))+\arctan(|h|/((1-s)\|x_1-x_2\|))\,.
$$
Using, for $x\geq 0$, the elementary inequality $x-x^2\leq \arctan x \leq x$, we deduce that
\begin{align*}
& \frac{|h|}{s(1-s)\|x_1-x_2\|}-\frac{h^2}{s^2(1-s)^2\|x_1-x_2\|^2}\\
& \quad\leq \arctan(|h|/(s\|x_1-x_2\|))+\arctan(|h|/((1-s)\|x_1-x_2\|)) \leq \frac{|h|}{s(1-s)\|x_1-x_2\|}\,.
\end{align*}
Consequently, $\pi-\theta(x_1,x_2,x_3(s,h))\leq yt^{-\gamma}$ is satisfied if
$$
|h| \leq s(1-s)\|x_1-x_2\| y t^{-\gamma}
$$
and cannot hold if
$$
|h| \geq s(1-s)\|x_1-x_2\| (y t^{-\gamma}+2 y^2 t^{-2\gamma})
$$
and $t$ is sufficiently large. Let $A_{y,t}$ be the set of all $x_1,x_2\in K$ such that
$$
B_{\tan(t^{-\gamma}y/2) \|x_1-x_2\|}^d(x_1),\,B_{\tan(t^{-\gamma}y/2) \|x_1-x_2\|}^d(x_2)\subset K\,.
$$
Now the previous considerations yield that, for $t$ sufficiently large and $(x_1,x_2)\in A_{y,t}$,
\begin{align*}
& \frac{t^3}{2}\int_0^1\int_{-\infty}^\infty  \I\{x_3(s,h)\in A(x_1,x_2,yt^{-\gamma}),x_3(s,h)={\rm LA}(x_1,x_2,x_3)\}\\
&\hskip 3cm \times \|x_1-x_2\|\, \varphi(x_1) \varphi(x_2)\varphi(x_3(s,h))\,\dint h \, \dint s\\
& = t^3 \int_0^1 \big(s(1-s)\|x_1-x_2\| y t^{-\gamma}+R(x_1,x_2,s)\big) \|x_1-x_2\| \\
& \hskip 3cm  \times\varphi(x_1) \varphi(x_2)\varphi(sx_1+(1-s)x_2) \, \dint s\\
& \quad +  \frac{t^3}{2}\int_0^1\int_{-\infty}^\infty  \I\{x_3(s,h)\in A(x_1,x_2,yt^{-\gamma}),x_3(s,h)={\rm LA}(x_1,x_2,x_3)\} \|x_1-x_2\|\, \\
& \hskip 3cm \times\varphi(x_1) \varphi(x_2)\big(\varphi(x_3(s,h))-\varphi(sx_1+(1-s)x_2)\big) \,\dint h \, \dint s
\end{align*}
with $R(x_1,x_2,s)$ satisfying the estimate $|R(x_1,x_2,s)|\leq 2 s(1-s)\|x_1-x_2\| y^2 t^{-2\gamma}$. For $(x_1,x_2)\not\in A_{y,t}$ the right hand-side is an upper bound. The choice $\gamma=3$ leads to
\begin{align*}
& |\alpha_t(y)-\beta y|\\
& \leq \int_{K^2\setminus A_{y,t}} \int_0^1 s(1-s) \|x_1-x_2\|^2 y \, \varphi(x_1) \varphi(x_2)\varphi(sx_1+(1-s)x_2) \, \dint s \, \dint(x_1,x_2)\\
& \quad +2 t^{-3} \int_{K^2} \int_0^1 s (1-s) y^2 \|x_1-x_2\|^2 \, \varphi(x_1) \varphi(x_2)\varphi(sx_1+(1-s)x_2) \, \dint s \, \dint(x_1,x_2)\\
& \quad + \frac{t^3}{2} \int_{K^2}\int_0^1\int_{-\infty}^\infty  \I\{x_3(s,h)\in A(x_1,x_2,yt^{-\gamma})\} \|x_1-x_2\|\, \varphi(x_1) \varphi(x_2)\\
& \hskip 3cm \times\big|\varphi(x_3(s,h))-\varphi(sx_1+(1-s)x_2)\big| \,\dint h \, \dint s \, \dint(x_1,x_2)\,.
\end{align*}
Note that $\ell_2^2(K\setminus A_{y,t})$ is of order $t^{-3}$ so that the first integral on the right-hand side is of the same order. By the Lipschitz continuity of the density $\varphi$ there is a constant $C_\varphi>0$ such that
$$
|\varphi(x_3(s,h))-\varphi(sx_1+(1-s)x_2)| \leq C_\varphi h\,.
$$
This implies that the third integral is of order $t^{-3}$. Combined with the fact that also the second integral above is of order $t^{-3}$, we see that there is a constant $C_{y,1}>0$ such that
$$
|\alpha_t(y)-\beta y| \leq C_{y,1} t^{-3}
$$
for $t\geq 1$.

For given $x_1,x_2\in K$, we have that
$$
\int_K \I\{x_3\in A(x_1,x_2,y t^{-\gamma})\} \, \varphi(x_3) \, \dint x_3 \leq M \int_K \I\{x_3\in A(x_1,x_2,yt^{-\gamma})\} \, \dint x_3
$$
with $M=\sup_{z\in K}\varphi(z)$. By the same arguments as above, we see that the integral over all $x_3$ such that the largest angle is adjacent to $x_3$ is bounded by
\begin{align*}
& M \int_0^1 s(1-s)\|x_1-x_2\|yt^{-3}+2s(1-s)\|x_1-x_2\|y^2t^{-6} \dint s\\
& \leq 2 M {\operatorname{diam}}(K) (yt^{-3}+2y^2t^{-6})\,,
\end{align*}
where ${\operatorname{diam}}(K)$ stands for the diameter of $K$.
The maximal angle is at $x_1$ or $x_2$ if $x_3$ is contained in the union of two cones with opening angle $2t^{-3}y$ and apices at $x_1$ and $x_2$, respectively. The integral over these $x_3$ is bounded by
$2 M {\operatorname{diam}}(K)^2t^{-3}y$. Altogether, we obtain that
\begin{align*}
& \int_K \I\{x_3\in A(x_1,x_2,y t^{-\gamma})\} \varphi(x_3) \, \dint x_3\\
& \leq 2 M  {\operatorname{diam}}(K) (yt^{-3}+2y^2t^{-6})+ 2 M {\operatorname{diam}}(K)^2y t^{-3}.
\end{align*}
This estimate implies that, for any $\ell\in\{1,2\}$,
\begin{align*}
& t^\ell \int_{K^\ell} \!\!\bigg( t^{3-\ell}\int_{K^{3-\ell}}\!\!\! \I\{x_3\in A(x_1,x_2,y t^{-3})\} \, \mu^{3-\ell}(\dint(K_{\ell+1},\hdots,K_3)) \bigg)^2 \!\!\! \mu^\ell(\dint(K_1,\hdots,K_\ell))\\
& \leq t^{6-\ell} (M\ell_2(K))^{4-\ell} \big( 2 M {\operatorname{diam}}(K) (yt^{-3}+2y^2t^{-6})+ 2 M {\operatorname{diam}}(K)^2 y t^{-3} \big)^2\,.
\end{align*}
Since the upper bound behaves like $t^{-\ell}$ for $t\geq 1$, there is a constant $C_{y,2}>0$ such that
$$
r_t(y)\leq C_{y,2} t^{-1}
$$
for $t\geq 1$. Now an application of Corollary \ref{corol:Main} concludes the proof in case of an underlying Poisson point process. The binomial case can be handled similarly using Corollary \ref{corol:MainBinomial}.
\end{proof}

\begin{remark}
We have assumed that the density $\varphi$ is Lipschitz continuous. If this is not the case, one can still show that the re-scaled point processes $t^3\xi_t$ and $n^3\widehat{\xi}_n$ converge in distribution to a homogeneous Poisson point process on $\RR_+$ with intensity $\beta$. However, we are then no more able to provide a rate of convergence for the associated order statistics $M_t^{(m)}$.
\end{remark}

\begin{remark}
In \cite[Section 5]{SilvermanBrown} the asymptotic behaviour of the number of flat triangles in a binomial point process has been investigated, while our focus here was on the angle statistic of such triangles. However, these two random variables are asymptotically equivalent so that Corollary \ref{cor:flattriangles} also delivers an alternative approach to the results in \cite{SilvermanBrown}. In addition, it allows to deal with an underlying Poisson point process, where it provides rates of convergence in the case of a Lipschitz density.
\end{remark}

\subsection{Non-intersecting $k$-flats}\label{subsec:Flats}

Fix a space dimension $d\geq 3$ and let $k\geq 1$ be such that $2k<d$. By $G(d,k)$ let us denote the space of $k$-dimensional linear subspaces of $\RR^d$, which is equipped with a probability measure $\varsigma$. In what follows we shall assume that $\varsigma$ is absolutely continuous with respect to the Haar probability measure on $G(d,k)$. The space of $k$-dimensional affine subspaces of $\RR^d$ is denoted by $A(d,k)$ and for $t>0$ a translation-invariant measure $\mu_t$ on $A(d,k)$ is defined by the relation
\begin{equation}\label{eq:DefMuFlats}
\int_{A(d,k)}g(E)\,\mu_t(\dint E)=t\int_{G(d,k)}\int_{L^\perp}g(L+x)\,\ell_{d-k}(\dint x)\,\varsigma(\dint L)\,,
\end{equation}
where $g\geq 0$ is a measurable function on $A(d,k)$. We will use $E$ and $F$ to indicate elements of $A(d,k)$, while $L$ and $M$ will stand for linear subspaces in $G(d,k)$. We also put $\mu=\mu_1$. For two fixed $k$-flats $E,F\in A(d,k)$ we denote by $d(E,F)=\inf\{\|x_1-x_2\|:x_1\in E,\,x_2\in F\}$ the distance of $E$ and $F$. For almost all $E$ and $F$ it is realized by two uniquely determined points $x_E\in E$ and $x_F\in F$, i.e. $d(E,F)=\|x_E-x_F\|$, and we let $m(E,F):=(x_E+x_F)/2$ be the midpoint of the line segment joining $x_E$ with $x_F$.

Let $K\subset\RR^d$ be a convex body and let $\eta_t$ be a Poisson point process on $A(d,k)$ with intensity measure $\mu_t$ as defined in \eqref{eq:DefMuFlats}. We will speak about $\eta_t$ as a Poisson $k$-flat process and denote, more generally, the elements of $A(d,k)$ or $G(d,k)$ as $k$-flats. We will not treat the binomial case in what follows since the measures $\mu_t$ are not finite. We notice that in view of \cite[Theorem 4.4.5 (c)]{SW} any two $k$-flats of $\eta_t$ are almost surely in general position, a fact which from now on will be used without further comment.

Point processes of $k$-dimensional flats in $\RR^d$ have a long tradition in stochastic geometry and we refer to \cite{SKM} or \cite{SW} for general background material. Moreover, we mention the works \cite{HugLastWeil} and \cite{SchneiderDuality}, which deal with distance measurements and the so-called proximity of Poisson $k$-flat processes and are close to what we consider here. While in these papers only mean values are considered, we are interested in the point process $\xi_t$ on $\RR_+$ defined by
$$\xi_t:={1\over 2}\sum_{(E,F)\in\eta_{t,\neq}^2}\delta_{d(E,F)^a}\,\I\{m(E,F)\in K\}$$
for a fixed parameter $a>0$. A particular case arises when $a=1$. Then $M_t^{(1)}$, for example, is the smallest distance between two $k$-flats from $\eta_t$ that have their midpoint in $K$.

\begin{corollary}\label{corol:Flats}
Define $$\beta=\frac{\ell_d(K)}{2}\,\kappa_{d-2k}\,\int_{G(d,k)}\int_{G(d,k)}[L,M]\,\varsigma(\dint L)\varsigma(\dint M)\,,$$ where $[L,M]$ is the $2k$-dimensional volume of a parallelepiped spanned by two orthonormal bases in $L$ and $M$. Then, as $t\to\infty$, $t^{2a/(d-2k)}\xi_t$ converges in distribution to a Poisson point process on $\RR_+$ with intensity measure $$B\mapsto(d-2k)\frac{\beta}{a}\int_Bu^{(d-2k-a)/a}\,\dint u\,,\qquad B\subset\RR_+\quad{\rm Borel}\,.$$ Moreover, there is a constant $C>0$ depending on $K$, $\varsigma$ and $a$ such that
\begin{equation*}
\begin{split}
\Bigg|\PP(t^{2a/(d-2k)}M_t^{(m)}>y) -& \exp\left(-\beta y^{(d-2k)/a}\right)\sum_{i=1}^{m-1}\frac{\left(\beta y^{(d-2k)/a}\right)^i}{i!}\Bigg|\\
&\leq C\, (y^{2(d-2k)/a}+y^{d-k+2(d-2k)/a})\,t^{-1}
\end{split}
\end{equation*}
for any $t\geq 1$, $y\geq 0$ and $m\in\NN$.
\end{corollary}

\begin{proof}
For $y>0$ and $t>0$ we have that
\begin{eqnarray*}
\alpha_t(y) &=& \frac{t^2}{2}\int_{A(d,k)}\int_{A(d,k)}\I\{d(E,F)\leq y^{1/a}t^{-\gamma/a},\,m(E,F)\in K\}\,\mu(\dint E)\mu(\dint F)\,.
\end{eqnarray*}
We abbreviate $\delta:=y^{1/a}t^{-\gamma/a}$ and evaluate the integral $$\cI:=\int_{A(d,k)}\int_{A(d,k)}\I\{d(E,F)\leq \delta,\,m(E,F)\in K\}\,\mu(\dint E)\mu(\dint F)\,.$$ For this, we define $V:=E+F$ and $U:=V^\bot$ and write $E$ and $F$ as $E=L+x_1$ and $F=M+x_2$ with $L,M\in G(d,k)$ and $x_1\in L^\bot$, $x_2\in M^\bot$. Applying now the definition \eqref{eq:DefMuFlats} of the measure $\mu$ and arguing along the lines of the proof of Theorem 4.4.10 in \cite{SW}, we arrive at the expression
\begin{equation*}
\begin{split}
\cI=\int_{G(d,k)}\int_{G(d,k)}\int_U\int_U &[L,M]\,\ell_{2k}\left(K\cap\left(V+\left(\frac{x_1+x_2}{2}\right)\right)\right) \\
&\times\I\{\|x_1-x_2\|\leq\d\} \,\ell_{d-2k}(\dint x_1)\ell_{d-2k}(\dint x_2) \varsigma(\dint L)\varsigma(\dint M)\,.
\end{split}
\end{equation*}
Substituting $u=x_1-x_2$, $v=(x_1+x_2)/2$ (a transformation having Jacobian equal to $1$), we find that
\begin{equation}\label{eq:ZwischenschrittI}
\begin{split}
\cI  = \int_{G(d,k)}\int_{G(d,k)}\int_U\int_U [L,M]\,&\ell_{2k}\big(K\cap(V+v)\big)\,\I\{\|u\|\leq\delta\}\\
&\ell_{d-2k}(\dint u)\ell_{d-2k}(\dint v)\varsigma(\dint L)\varsigma(\dint M)\,.
\end{split}
\end{equation}
Since $U$ has dimension $d-2k$, transformation into spherical coordinates in $U$ gives
\begin{eqnarray*}
\int_U{\bf 1}(\|u\|\leq\d)\,\dint u &=& (d-2k)\kappa_{d-2k}\int_0^\delta r^{d-2k-1}\,\dint r=\kappa_{d-2k}\d^{d-2k}\,.
\end{eqnarray*}
Moreover, $$\int_U\ell_{2k}\big(K\cap(V+v)\big)\,\ell_{d-2k}(\dint v)=\ell_d(K)$$ since $V=U^\perp$.  Combining these facts with \eqref{eq:ZwischenschrittI} we find that $$\cI=\delta^{d-2k}\,\ell_d(K)\,\kappa_{d-2k}\,\int_{G(d,k)}\int_{G(d,k)}[L,M]\,\varsigma(\dint L)\varsigma(\dint M)\,$$ and that
\begin{equation*}
\begin{split}
\alpha_t(y) =\frac{1}{2}\,\ell_d(K)\,\kappa_{d-2k}\,y^{(d-2k)/a}\,t^{2-\gamma(d-2k)/a}\int_{G(d,k)}\int_{G(d,k)}[L,M]\,\varsigma(\dint L)\varsigma(\dint M)\,.
\end{split}
\end{equation*}
 Consequently, choosing $\gamma=2a/(d-2k)$ we have that $$\alpha_t(y)=\beta y^{(d-2k)/a}\,.$$ For the remainder term $r_t(y)$ we write
$$
r_t(y)=t\int_{A(d,k)} \bigg(t\int_{A(d,k)}\I\{d(E,F)^a\leq yt^{-\gamma},\,m(E,F)\in K\}\,\mu(\dint F) \bigg)^2 \mu(\dint E)\,.
$$
This can be estimated along the lines of the proof of Theorem 3 in \cite{STFlats}. Namely, using that $[\,\cdot\,,\,\cdot\,]\leq 1$ and writing ${\rm diam}(K)$ for the diameter of $K$, we find that
\begin{eqnarray*}
r_t(y) &\leq& t\kappa_{d-k} ({\rm diam}(K)+2t^{-\gamma}y)^{d-k} \int_{G(d,k)} \bigg( t\int_{G(d,k)}\int_{(L+M)^\bot}\I\{\|x\|^a\leq yt^{-\gamma}\}\\
&&\qquad\qquad\times\kappa_k({\rm diam}(K)/2)^k\,\ell_{d-2k}(\dint x)\varsigma(\dint M) \bigg)^2  \, \varsigma(\dint L)\\
&\leq &  t\kappa_{d-k} ({\rm diam}(K)+2t^{-\gamma}y)^{d-k} \big(t \kappa_{d-2k} (yt^{-\gamma})^{(d-2k)/a} \kappa_k ({\rm diam}(K)/2)^k\big)^2\\
&=& \kappa_{d-k} ({\rm diam}(K)+2t^{-2a/(d-2k)}y)^{d-k}  \kappa_{d-2k}^2  \kappa_k^2 ({\rm diam}(K)/2)^{2k} \,y^{2(d-2k)/a}\,t^{-1}\,,
\end{eqnarray*}
where we have used that $\gamma=2a/(d-2k)$. This puts us in the position to apply Corollary \ref{corol:Main}, which completes the proof.\hfill $\Box$
\end{proof}

\begin{remark}
A particularly interesting case arises when the distribution $\varsigma$ coincides with the Haar probability measure on $G(d,k)$. Then the double integral in the definition of $\beta$ in Corollary \ref{corol:Flats} can be evaluated explicitly, namely we have
\begin{equation*}
\int_{G(d,k)}\int_{G(d,k)}[L,M]\,\varsigma(\dint L)\varsigma(\dint M)=\frac{{d-k\choose k}\k_{d-k}^2}{{d\choose k}\k_d\k_{d-2k}}
\end{equation*}
according to \cite[Lemma 4.4]{HugSchneiderSchuster2008}.
\end{remark}

\begin{remark}
Corollary \ref{corol:Flats} generalizes Theorem 4 in \cite{STFlats} (where the case $a=1$ has been investigated) to general length-powers $a>0$. However, it should be noticed that the set-up in \cite{STFlats} slightly differs from the one here. In \cite{STFlats} the intensity parameter $t$ was kept fixed, whereas the set $K$ was increased by dilations. But because of the scaling properties of a Poisson $k$-flat process and the $a$-homogeneity of $d(E,F)^a$, one can translate one result into the other. Moreover, we refer to \cite{HTW} for closely related results including directional constraints.
\end{remark}

\begin{remark}
In \cite{STScalingLimits} a similar problem has been addressed in the case where $\varsigma$ coincides with the Haar probability measure on $G(d,k)$. For a pair $(E,F)\in\eta_{t,\neq}^2$ satisfying $E\cap K\neq\emptyset$ and $F\cap K\neq\emptyset$, the distance between $E$ and $F$ was measured by $$d_K(E,F)=\inf\{\|x_1-x_2\|:x_1\in E\cap K,\,x_2\in F\cap K\},$$ and it has been shown in Theorem 2.1 ibidem that the associated point process 
$$\xi_t:={1\over 2}\sum_{(E,F)\in\eta_{t,\neq}^2}\delta_{d_K(E,F)}\,\I\{E\cap K\neq\emptyset,\,F\cap K\neq\emptyset\}$$
converges, after rescaling with $t^{2/(d-2k)}$, towards the same Poisson point process as in Corollary \ref{corol:Flats} when $\varsigma$ is the Haar probability measure on $G(d,k)$ and $a=1$.
\end{remark}

\section{Proofs of the main results}\label{sec:Proofs}


\subsection{Moment formulas for Poisson U-statistics}\label{subsec:Moments}

We call a Poisson functional $S$ of the form
$$
S=\sum_{(x_1,\hdots,x_k)\in\eta_{t,\neq}^k} f(x_1,\hdots,x_k)
$$
with $k\in\NN_0:=\NN\cup\{0\}$ and $f: \XX^k\to\RR$ a U-statistic of order $k$ of $\eta_t$, or a Poisson U-statistic for short (see \cite{LachiezeReyReitznerChapter}). For $k=0$ we use the convention that $f$ is a constant and $S=f$. In the following, we always assume that $f$ is integrable. Moreover, without loss of generality we assume that $f$ is symmetric since we sum over all permutations of a fixed $k$-tuple of points in the definition of $S$.

In order to compute mixed moments of Poisson U-statistics, we use the following notation. For $\ell\in\NN$ and $n_1,\hdots,n_\ell\in\NN_0$ we define $N_0=0$, $N_i=\sum_{j=1}^{i}n_j,  \ i\in\{1,\hdots,\ell\}$, and
$$
J_i=\begin{cases} \{N_{i-1}+1,\hdots, N_i\}, & N_{i-1}<N_i\\ \emptyset, & N_{i-1}=N_i \end{cases}, \quad i\in\{1,\hdots,\ell\}.
$$
Let $\Pi(n_1,\hdots,n_\ell)$ be the set of all partitions $\sigma$ of $\{1,\hdots,N_\ell\}$ such that for any $i\in\{1,\hdots,\ell\}$ all elements of $J_i$ are in different blocks of $\sigma$. By $|\sigma|$ we denote the number of blocks of $\sigma$. We say that two blocks $B_1$ and $B_2$ of a partition $\sigma\in\Pi(n_1,\hdots,n_\ell)$ intersect if there is an $i\in\{1,\hdots,\ell\}$ such that $B_1\cap J_i\neq \emptyset$ and $B_2\cap J_i \neq \emptyset$. A partition $\sigma\in\Pi(n_1,\hdots,n_\ell)$ with blocks $B_1,\hdots,B_{|\sigma|}$ belongs to $\widetilde{\Pi}(n_1,\hdots,n_\ell)$ if there are no non-empty sets $M_1,M_2\subset\{1,\hdots,|\sigma|\}$ with $M_1\cap M_2=\emptyset$ and $M_1\cup M_2=\{1,\hdots,|\sigma|\}$ such that for any $i\in M_1$ and $j\in M_2$ the blocks $B_i$ and $B_j$ do not intersect. Moreover, we define
$$
\Pi_{\neq}(n_1,\hdots,n_\ell)=\{\sigma\in\Pi(n_1,\hdots,n_\ell): |\sigma|>\min\{n_1,\hdots,n_\ell\}\}.
$$
If there are $i,j\in\{1,\hdots,\ell\}$ with $n_i\neq n_j$, we have $\Pi_{\neq}(n_1,\hdots,n_\ell)=\Pi(n_1,\hdots,n_\ell)$.

For $\sigma\in\Pi(n_1,\hdots,n_\ell)$ and $f: \XX^{N_\ell} \to \RR$ we define $f_\sigma: \XX^{|\sigma|}\to \RR$ as the function which arises by replacing in the arguments of $f$ all variables belonging to the same block of $\sigma$ by a new common variable. Since we are only interested in the integral of this new function in the sequel, the order of the new variables does not matter. For $f^{(i)}: \XX^{n_i}\to\RR$, $i\in\{1,\hdots,\ell\}$, let $\otimes_{i=1}^\ell f^{(i)}: \XX^{N_\ell}\to\RR$ be given by
$$
\big(\otimes_{i=1}^\ell f^{(i)}\big)(x_1,\hdots,x_{N_\ell}) = \prod_{i=1}^\ell f^{(i)}(x_{N_{i-1}+1},\hdots,x_{N_i})\,.
$$
The following lemma allows us to compute moments of Poisson U-statistics (see also \cite{PrivaultChapter}). Here and in what follows we mean by a Poisson functional $F=F(\eta_t)$ a random variable only depending on the Poisson point process $\eta_t$ for some fixed $t>0$.

\begin{lemma}\label{lem:momentsUstatistics}
For $\ell\in\NN$ and $f^{(i)}\in L^1_s(\mu_t^{k_i})$ with $k_i\in\NN_0$, $i=1,\hdots,\ell$, such that
$$
\int_{\XX^{|\sigma|}} |\big(\otimes_{i=1}^\ell f^{(i)}\big)_\sigma| \, \dint\mu_t^{|\sigma|}<\infty\qquad{\rm for\ all}\qquad \sigma\in\Pi(k_1,\hdots,k_\ell)\,,
$$
let
$$
S_i=\sum_{(x_1,\hdots,x_{k_i})\in\eta^{k_i}_{t,\neq}} f^{(i)}(x_1,\hdots,x_{k_i}), \quad i=1,\hdots,\ell\,,
$$
and let $F$ be a bounded Poisson functional. Then
\begin{equation*}
\begin{split}
 \EE\Big[ F \prod_{i=1}^\ell S_i \Big]= \sum_{\sigma\in\Pi(k_1,\hdots,k_\ell)} & \int_{\XX^{|\sigma|}} \big(\otimes_{i=1}^\ell f^{(i)}\big)_\sigma(x_1,\hdots,x_{|\sigma|})\\
&  \times \EE[F(\eta_t+\sum_{i=1}^{|\sigma|}\delta_{x_i})] \,\mu_t^{|\sigma|}(\dint(x_1,\hdots,x_{|\sigma|}))\,.
\end{split}
\end{equation*}
\end{lemma}

\begin{proof}
We can rewrite the product as
\begin{equation*}
\begin{split}
& F(\eta_t) \prod_{i=1}^\ell \sum_{(x_1,\hdots,x_{k_i})\in\eta_{t,\neq}^{k_i}} f^{(i)}(x_1,\hdots,x_{k_i})\\
& = \sum_{\sigma\in\Pi(k_1,\hdots,k_\ell)} \sum_{(x_1,\hdots,x_{|\sigma|})\in\eta_{t,\neq}^{|\sigma|}} \big(\otimes_{i=1}^\ell f^{(i)}\big)_\sigma(x_1,\hdots,x_{|\sigma|}) F(\eta_t)
\end{split}
\end{equation*}
since points occurring in different sums on the left-hand side can be either equal or distinct. Now an application of the multivariate Mecke formula (see \cite{LastChapter}) completes the proof of the lemma. \hfill $\Box$
\end{proof}

\subsection{Poisson approximation of Poisson U-statistics}\label{subsec:PoissonLimit}

The key argument of the proof of Theorem \ref{thm:Main} is a quantitative bound for the Poisson approximation of Poisson U-statistics which is established in this subsection. From now on we consider the Poisson U-statistic
$$
S_A=\frac{1}{k!} \sum_{(x_1,\hdots,x_k)\in\eta^k_{t,\neq}} \I\{f(x_1,\hdots,x_k)\in A\}\,,
$$
where $f$ is as in Section \ref{sec:Results} and $A\subset\mathbb{R}$ is measurable and bounded. We assume that $k\geq 2$ since $S_A$ follows a Poisson distributon for $k=1$ (see Section 2.3 in \cite{Kingman}, for example). In the sequel, we use the abbreviation
$$
h(x_1,\hdots,x_k):=\frac{1}{k!} \I\{f(x_1,\hdots,x_k)\in A\}, \quad x_1,\hdots,x_k\in \XX\,.
$$
It follows from the multivariate Mecke formula (see \cite{LastChapter}) that
$$
s_A:=\EE[S_A]=\int_{\XX^k}h(x_1,\hdots,x_k) \, \mu_t^k(\dint(x_1,\hdots,x_k))\,.
$$
In order to compare the distributions of two integer-valued random variables $Y$ and $Z$, we use the so-called total variation distance $d_{TV}$ defined by
$$
d_{TV}(Y,Z)=\sup_{B\subset\mathbb{Z}}\big|\PP(Y\in B)-\PP(Z\in B)\big|\,.
$$

\begin{proposition}\label{prop:Poissonlimit}
Let $S_A$ be as above, let $Y$ be a Poisson distributed random variable with mean $s>0$ and define
$$
\varrho_A:= \max_{1\leq\ell\leq k-1} \int_{\XX^\ell} \bigg( \int_{\XX^{k-\ell}} h(x_1,\hdots,x_k) \, \mu_t^{k-\ell}(\dint(x_{\ell+1},\hdots,x_k)) \bigg)^2\mu_t^\ell(\dint(x_1,\hdots,x_\ell)).
$$
Then there is a constant $C\geq 1$ only depending on $k$ such that
\begin{equation}\label{eq:PropdTVBound}
d_{TV}(S_A,Y) \leq |s_A-s|+C\,\min\left\{1,{1\over s_A}\right\}\varrho_A\,.
\end{equation}
\end{proposition}

\begin{remark}
The inequality \eqref{eq:PropdTVBound} still holds if $Y$ is almost surely zero (such a $Y$ can be interpreted as a Poisson distributed random variable with mean $s=0$). In this case, we obtain by Markov's inequality that
$$
d_{TV}(S_A,Y) =\PP(S_A\geq 1) \leq \EE S_A=s_A.
$$

\end{remark}

Our proof of Proposition \ref{prop:Poissonlimit} is a modification of the proof of Theorem 3.1 in \cite{Peccati2011}. It makes use of the special structure of $S_A$ and improves of the bound in \cite{Peccati2011} in case of Poisson U-statistics. To prepare for what follows, we need to introduce some facts around the Chen-Stein method for Poisson approximation (compare with \cite{BourgPec}). For a function $f:\mathbb{N}_0\to\mathbb{R}$ let us define $\Delta f(k):=f(k+1)-f(k)$, $k\in\mathbb{N}_0$, and $\Delta^2f(k):=f(k+2)-2f(k+1)+f(k)$, $k\in\mathbb{N}_0$. For $B\subset \mathbb{N}_0$ let $f_B$ be the solution of the Chen-Stein equation
\begin{equation}\label{eqn:ChenStein}
\I\{k\in B\}-\PP(Y\in B)=s f(k+1)-kf(k), \quad k\in\mathbb{N}_0\,.
\end{equation}
It is known (see Lemma 1.1.1 in \cite{BarbourEtAl}) that $f_B$ satisfies
\begin{equation}\label{eq:PropertiesfB}
\|f_B\|_\infty \leq 1 \quad \text{and} \quad \|\Delta f_B\|_\infty\leq \min\left\{1,{1\over s}\right\}=:\varepsilon_1\,,
\end{equation}
where $\|\cdot\|_\infty$ is the usual supremum norm.

Besides the Chen-Stein method we need some facts concerning the Malliavin calculus of variations on the Poisson space (see \cite{LastChapter}). First, the so-called integration by parts formula implies that
\begin{equation}\label{eqn:IntegrationByParts}
\EE[f_B(S_A)(S_A-\EE[S_A])]=\EE\int_\XX D_xf_B(S_A) (-D_xL^{-1}S_A) \, \mu_t(\dint x)\,,
\end{equation}
where $D$ stands for the difference operator and $L^{-1}$ is the inverse of the Ornstein-Uhlenbeck generator (this step requires that $\EE\int_\XX(D_xS_A)^2\,\mu_t(\dint x)<\infty$, which is a consequence of the calculations in the proof of Proposition \ref{prop:Poissonlimit}). The following lemma (see Lemma 3.3 in \cite{ReitznerSchulte2013}) implies that the difference operator applied to a Poisson U-statistic leads again to a Poisson U-statistic.

\begin{lemma}\label{lem:D}
Let $k\in\NN$, $f\in L^1_s(\mu_t^k)$ and
$$
S=\sum_{(x_1,\hdots,x_k)\in\eta^k_{t,\neq}} f(x_1,\hdots,x_k)\,.
$$
Then
$$
D_xS=k \sum_{(x_1,\hdots,x_{k-1})\in\eta^{k-1}_{t,\neq}} f(x,x_1,\hdots,x_{k-1})\,, \quad x\in \XX\,.
$$
\end{lemma}
\begin{proof}
It follows from the definition of the difference operator and the assumption that $f$ is a symmetric function that
\begin{align*}
D_xS & = \sum_{(x_1,\hdots,x_k)\in(\eta_t+\delta_x)^k_{\neq}} f(x_1,\hdots,x_k) - \sum_{(x_1,\hdots,x_k)\in\eta^k_{t,\neq}} f(x_1,\hdots,x_k)\\
& = \sum_{(x_1,\hdots,x_{k-1})\in\eta^{k-1}_{t,\neq}} \big(f(x,x_1,\hdots,x_{k-1}) + \hdots + f(x_1,\hdots,x_{k-1},x)\big)\\
& = k\sum_{(x_1,\hdots,x_{k-1})\in\eta^{k-1}_{t,\neq}} f(x,x_1,\hdots,x_{k-1})
\end{align*}
for $x\in \XX$. This completes the proof.\hfill $\Box$
\end{proof}

In order to derive an explicit formula for the combination of the difference operator and the inverse of the Ornstein-Uhlenbeck generator of $S_A$, we define $h_\ell: \XX^\ell\to\mathbb{R}$, $\ell\in\{1,\hdots,k\}$, by
$$
h_\ell(x_1,\hdots,x_{\ell}):=\int_{\XX^{k-\ell}} h(x_1,\hdots,x_{\ell},\hat{x}_{1},\hdots,\hat{x}_{k-\ell}) \, \mu_t^{k-\ell}(\dint(\hat{x}_1,\hdots,\hat{x}_{k-\ell}))\,.
$$
We shall see now that the operator $-DL^{-1}$ applied to $S_A$ can be expressed as a sum of Poisson U-statistics (see also Lemma 5.1 in \cite{SchulteKolmogorov}).

\begin{lemma}\label{lem:DL-1}
For $x\in \XX$,
$$
-D_xL^{-1}S_A = \sum_{\ell=1}^{k}\sum_{(x_1,\hdots,x_{\ell-1})\in\eta^{\ell-1}_{t,\neq}} h_\ell(x,x_1,\hdots,x_{\ell-1})\,.
$$
\end{lemma}

\begin{proof}
By Mehler's formula (see Theorem 3.2 in \cite{LPS} and also \cite{LastChapter}) we have
$$
-L^{-1}S_A = \int_0^1\int \frac{1}{s} \EE\big[ \sum_{(x_1,\hdots,x_k)\in(\eta_t^{(s)}+\chi)^{k}_{\neq}} h(x_1,\hdots,x_k) -s_A  \big| \eta_t \big] \PP_{(1-s)\mu_t}(\dint \chi) \, \dint s
$$
where $\eta_t^{(s)}$, $s\in[0,1]$, is an $s$-thinning of $\eta_t$ and $\PP_{(1-s)\mu_t}$ is the distribution of a Poisson point process with intensity measure $(1-s)\mu_t$. Note in particular that $\eta_t^{(s)}+\chi$ is a Poisson point process with intensity measure $s\mu_t+(1-s)\mu_t=\mu_t$. The last expression can be rewritten as
\begin{align*}
-L^{-1}S_A & = \int_0^1\int \frac{1}{s} \EE\big[ \sum_{(\hat{x}_1,\hdots,\hat{x}_k)\in\chi_{\neq}^k} h(\hat{x}_1,\hdots,\hat{x}_k) -s_A \big| \eta_t \big] \PP_{(1-s)\mu_t}(\dint\chi) \, \dint s \allowdisplaybreaks\\
& \quad +\sum_{\ell=1}^k \binom{k}{\ell}\int_0^1\int \frac{1}{s} \EE\big[ \sum_{(x_1,\hdots,x_\ell)\in(\eta_t^{(s)})^\ell_{\neq}}\\
& \hskip 1cm \sum_{(\hat{x}_1,\hdots,\hat{x}_{k-\ell})\in\chi_{\neq}^{k-\ell}} h(x_1,\hdots,x_\ell,\hat{x}_1,\hdots,\hat{x}_{k-\ell})  \big| \eta_t \big] \, \PP_{(1-s) \mu_t}(\dint \chi) \, \dint s\,. \allowdisplaybreaks
\end{align*}
By the multivariate Mecke formula (see \cite{LastChapter}), we obtain for the first term that
\begin{equation*}
\begin{split}
& \int_0^1\int \frac{1}{s} \EE\big[ \sum_{(\hat{x}_1,\hdots,\hat{x}_k)\in\chi_{\neq}^k} h(\hat{x}_1,\hdots,\hat{x}_k) -s_A \big| \eta_t \big] \, \PP_{(1-s)\mu_t}(\dint\chi) \, \dint s\\
& = \int_0^1 \int \frac{1}{s} \bigg( \sum_{(\hat{x}_1,\hdots,\hat{x}_k)\in\chi_{\neq}^k} h(\hat{x}_1,\hdots,\hat{x}_k) -s_A \bigg) \, \PP_{(1-s)\mu_t}(\dint\chi) \, \dint s = \int_0^1 \frac{(1-s)^k-1}{s} \, \dint s \ s_A\,.
\end{split}
\end{equation*}
To evaluate the second term further, we notice that for an $\ell$-tuple $(x_1,\ldots,x_\ell)\in\eta_{t,\neq}^\ell$ the probability of surviving the $s$-thinning procedure is $s^\ell$. Thus
\begin{align*}
\EE\big[ &\sum_{(x_1,\hdots,x_\ell)\in(\eta_t^{(s)})^\ell_{\neq}}\sum_{(\hat{x}_1,\hdots,\hat{x}_{k-\ell})\in\chi_{\neq}^{k-\ell}} h(x_1,\hdots,x_\ell,\hat{x}_1,\hdots,\hat{x}_{k-\ell})  \big| \eta_t \big]\\
& = s^\ell\,\sum_{(x_1,\hdots,x_\ell)\in\eta_{t,\neq}^\ell}\sum_{(\hat{x}_1,\hdots,\hat{x}_{k-\ell})\in\chi_{\neq}^{k-\ell}} h(x_1,\hdots,x_\ell,\hat{x}_1,\hdots,\hat{x}_{k-\ell})
\end{align*}
for $\ell\in\{1,\hdots,k\}$. This leads to
\begin{align*}
 &-L^{-1}S_A  = \int_0^1 \frac{(1-s)^k-1}{s} \, \dint s \ s_A\\
& +\sum_{\ell=1}^k \binom{k}{\ell}\int_0^1\int s^{\ell-1} \hskip -0.6cm \sum_{(x_1,\hdots,x_\ell)\in\eta_{t,\neq}^\ell} \sum_{(\hat{x}_1,\hdots,\hat{x}_{k-\ell})\in\chi_{\neq}^{k-\ell}} \hskip -0.3cm
 h(x_1,\hdots,x_\ell,\hat{x}_1,\hdots,\hat{x}_{k-\ell}) \, \PP_{(1-s) \mu_t}(\dint \chi) \, \dint s \allowdisplaybreaks\,.
\end{align*}
Finally, we may interpret $\chi$ as $(1-s)$-thinning of an independent copy of $\eta_t$, in which each point has survival probability $(1-s)$. Then the multivariate Mecke formula (see \cite{LastChapter}) implies that
\begin{align*}
-L^{-1}S_A  = &\int_0^1 \frac{(1-s)^k-1}{s} \, \dint s \ s_A \\
&+\sum_{\ell=1}^k \binom{k}{\ell}\int_0^1 s^{\ell-1}(1-s)^{k-\ell} \,\dint s \sum_{(x_1,\hdots,x_\ell)\in\eta_{t,\neq}^\ell}  h_\ell(x_1,\hdots,x_\ell) \, .
\end{align*}
Together with
$$
\int_0^1 s^{\ell-1} (1-s)^{k-\ell} \, \dint s = \frac{(\ell-1)! (k-\ell)!}{k!}, \quad \ell\in\{1,\hdots,k\},
$$
we see that
$$
-L^{-1}S_A = s_A\int_0^1 \frac{(1-s)^k-1}{s}  \, \dint s+\sum_{\ell=1}^k \frac{1}{\ell} \sum_{(x_1,\hdots,x_\ell)\in\eta_{t,\neq}^\ell} h_\ell(x_1,\hdots,x_\ell)\,.
$$
Applying now the difference operator to the last equation, we see that the first term does not contribute, whereas the second term can be handled by using Lemma \ref{lem:D}.\hfill $\Box$
\end{proof}

Now we are prepared for the proof of Proposition \ref{prop:Poissonlimit}.

\begin{proof}[of Proposition \ref{prop:Poissonlimit}]
Let $Y_A$ be a Poisson distributed random variable with mean $s_A>0$. The triangle inequality for the total variation distance implies that
$$
d_{TV}(S_A,Y) \leq d_{TV}(Y,Y_A)+d_{TV}(Y_A,S_A)\,.
$$
A standard calculation shows that
$$
d_{TV}(Y,Y_A)\leq |s-s_A|
$$
so that it remains to bound
$$
d_{TV}(Y_A,S_A)=\sup_{B\subset \mathbb{N}_0} |\PP(S_A\in B)-\PP(Y_A\in B)|\,.
$$
For a fixed $B\subset\mathbb{N}_0$ it follows from \eqref{eqn:ChenStein} and \eqref{eqn:IntegrationByParts} that
\begin{equation}\label{eqn:ExpectationChenStein}
\begin{split}
\PP(S_A\in B)-\PP(Y_A\in B) & = \EE[s_A \Delta f_B(S_A)-(S_A-s_A) f_B(S_A)] \\
& = \EE\big[s_A \Delta f_B(S_A)-\int_\XX D_xf_B(S_A) (-D_xL^{-1}S_A) \, \mu_t(\dint x) \big].
\end{split}
\end{equation}
Now a straightforward computation using a discrete Taylor-type expansion as in \cite{Peccati2011} shows that
\begin{align*}
D_xf_B(S_A) & = f_B(S_A+D_xS_A)-f_B(S_A)\\
& = \sum_{k=1}^{D_xS_A} \big(f_B(S_A+k)-f_B(S_A+k-1)\big)\\
& = \sum_{k=1}^{D_xS_A} \Delta f_B(S_A+k-1)\\
& = \Delta f_B(S_A) D_xS_A + \sum_{k=2}^{D_xS_A} \big(\Delta f_B(S_A+k-1)-\Delta f_B(S_A)\big).
\end{align*}
Together with \eqref{eq:PropertiesfB}, we obtain that
\begin{align*}
\bigg| \sum_{k=2}^{D_xS_A} \big(\Delta f_B(S_A+k-1)-\Delta f_B(S_A)\big) \bigg| & \leq 2\|\Delta f_B\|_\infty \max\{0, D_xS_A-1\}\\
& \leq 2\varepsilon_{1,A}\max\{0, D_xS_A-1\}
\end{align*}
with
$$
\varepsilon_{1,A}:=\min\Big\{1,{1\over s_A}\Big\}\,.
$$
Hence, we have
$$
D_xf_B(S_A) = \Delta f_B(S_A) D_xS_A+R_x ,
$$
where the remainder term satisfies $|R_x|\leq 2\varepsilon_{1,A}\max\{0,D_xS_A-1\}$. Together with \eqref{eqn:ExpectationChenStein} and $-D_xL^{-1}S_A\geq 0$, which follows from Lemma \ref{lem:DL-1}, we obtain that
\begin{equation}\label{eqn:BoundPoissonApproximation}
\begin{split}
& |\PP(S_A\in B)-\PP(Y_A\in B)|\\
& \leq \big|\EE\big[s_A \Delta f_B(S_A) - \Delta f_B(S_A) \int_\XX D_xS_A (-D_xL^{-1}S_A)\, \mu_t(\dint x)\big]\big|\\
& \qquad + 2\varepsilon_{1,A} \int_\XX \EE[ \max\{0,D_xS_A-1\}(-D_xL^{-1}S_A)] \, \mu_t(\dint x)\,.
\end{split}
\end{equation}
It follows from Lemma \ref{lem:D} and Lemma \ref{lem:DL-1} that
\begin{align*}
& \EE\big[\Delta f_B(S_A) \int_\XX D_xS_A (-D_xL^{-1}S_A)\, \mu_t(\dint x)\big]\\
& = \EE\big[\Delta f_B(S_A(\eta_t)) \int_\XX \Big(k\sum_{(x_1,\hdots,x_{k-1})\in\eta^{k-1}_{t,\neq}} h(x,x_1,\hdots,x_{k-1})\Big)\\
& \hskip 4cm \times\Big(\sum_{\ell=1}^k \sum_{(x_1,\hdots,x_{\ell-1})\in\eta^{\ell-1}_{t,\neq}} h_\ell(x,x_1,\hdots,x_{\ell-1})\Big) \, \mu_t(\dint x)\big]\,.
\end{align*}
Consequently, we can deduce from Lemma \ref{lem:momentsUstatistics} that
\begin{align*}
& \EE\big[\Delta f_B(S_A) \int_\XX D_xS_A (-D_xL^{-1}S_A)\, \mu_t(\dint x)\big]\\
& = k \sum_{\ell=1}^k \sum_{\sigma\in\Pi(k-1,\ell-1)} \int_{\XX^{|\sigma|+1}} \EE\big[\Delta f_B(S_A(\eta_t+\sum_{i=1}^{|\sigma|}\delta_{x_i}))]\\
& \hskip 3cm \big(h(x,\cdot) \otimes h_\ell(x,\cdot)\big)_\sigma(x_1,\hdots,x_{|\sigma|}) \, \mu_t^{|\sigma|+1}(\dint(x,x_1,\hdots,x_{|\sigma|}))\,.
\end{align*}
For the particular choice $\ell=k$ and $|\sigma|=k-1$ we have
\begin{align*}
& \int_{\XX^{|\sigma|+1}} \EE\big[\Delta f_B(S_A(\eta_t+\sum_{i=1}^{|\sigma|}\delta_{x_i}))] \big(h(x,\cdot) \otimes h_\ell(x,\cdot) \big)_\sigma(x_1,\hdots,x_{|\sigma|}) \\
& \hskip 1cm \mu_t^{|\sigma|+1}(\dint(x,x_1,\hdots,x_{|\sigma|})) \allowdisplaybreaks\\
& = \frac{1}{k!}\int_{\XX^k} \EE\big[\Delta f_B(S_A(\eta_t+\sum_{i=1}^{k-1}\delta_{x_i}))] \, h(x_1,\hdots,x_k) \, \mu_t^{k}(\dint(x_1,\hdots,x_k)) \allowdisplaybreaks\\
& = \frac{1}{k!} \int_{\XX^k} \EE\big[\Delta f_B(S_A(\eta_t+\sum_{i=1}^{k-1}\delta_{x_i}))-\Delta f_B(S_A(\eta_t))] \, h(x_1,\hdots,x_k) \, \mu_t^{k}(\dint(x_1,\hdots,x_k))\\
& \qquad + \frac{1}{k!} \int_{\XX^k} \EE\big[\Delta f_B(S_A(\eta_t))] \, h(x_1,\hdots,x_k) \, \mu_t^{k}(\dint(x_1,\hdots,x_k)) \allowdisplaybreaks\\
& = \frac{1}{k!} \int_{\XX^k} \EE\big[\Delta f_B(S_A(\eta_t+\sum_{i=1}^{k-1}\delta_{x_i}))-\Delta f_B(S_A(\eta_t))] \, h(x_1,\hdots,x_k) \, \mu_t^{k}(\dint(x_1,\hdots,x_k))\\
& \qquad + \frac{1}{k!} \EE\big[\Delta f_B(S_A)] s_A\,.
\end{align*}
Since there are $(k-1)!$ partitions $\sigma\in\Pi(k-1,k-1)$ with $|\sigma|=k-1$, we obtain that
\begin{align*}
& \big|\EE\big[s_A \Delta f_B(S_A) - \Delta f_B(S_A) \int_\XX D_xS_A (-D_xL^{-1}S_A)\, \mu_t(\dint x)\big]\big|\\
& \leq k \sum_{\ell=1}^k \sum_{\sigma\in \Pi_{\neq}(k-1,\ell-1)}\int_{\XX^{|\sigma|+1}} |\,\EE\big[\Delta f_B(S_A(\eta_t+\sum_{i=1}^{|\sigma|}\delta_{x_i}))]\,|\\
& \hskip 3cm \big(h(x,\cdot) \otimes h_\ell(x,\cdot) \big)_\sigma(x_1,\hdots,x_{|\sigma|}) \, \mu_t^{|\sigma|+1}(\dint(x,x_1,\hdots,x_{|\sigma|}))\\
& \quad  +\int_{\XX^k} \big|\EE\big[\Delta f_B(S_A(\eta_t+\sum_{i=1}^{k-1}\delta_{x_i}))-\Delta f_B(S_A(\eta_t))]\big| \, h(x_1,\hdots,x_k) \, \mu_t^{k}(\dint(x_1,\hdots,x_k))\,.
\end{align*}
Now \eqref{eq:PropertiesfB} and the definition of $\varrho_A$ imply that, for $\ell\in\{1,\hdots,k\}$,
\begin{equation*}
\begin{split}
& \sum_{\sigma\in\Pi_{\neq}(k-1,\ell-1)} \int_{\XX^{|\sigma|+1}} \,|\EE\big[\Delta f_B(S_A(\eta_t+\sum_{i=1}^{|\sigma|}\delta_{x_i}))]\,|\\
& \quad \quad \quad \quad \quad \quad  \big(h(x,\cdot) \otimes h_\ell(x,\cdot) \big)_\sigma(x_1,\hdots,x_{|\sigma|}) \, \mu_t^{|\sigma|+1}(\dint(x,x_1,\hdots,x_{|\sigma|}))\\
 & \leq \varepsilon_{1,A}\, |\Pi_{\neq}(k-1,\ell-1)| \varrho_A\,.
\end{split}
\end{equation*}
Hence, the first summand above is bounded by
$$
k\varepsilon_{1,A}\sum_{\ell=1}^k |\Pi_{\neq}(k-1,\ell-1)|\varrho_A\,.$$
By \eqref{eq:PropertiesfB} we see that
$$
\big|\EE\big[\Delta f_B(S_A(\eta_t+\sum_{i=1}^{k-1}\delta_{x_i}))-\Delta f_B(S_A(\eta_t))]\big| \leq 2\varepsilon_{1,A}\EE[S_A(\eta_t+\sum_{i=1}^{k-1}\delta_{x_i})-S_A(\eta_t)]\,,
$$
and the multivariate Mecke formula for Poisson point processes (see \cite{LastChapter}) leads to
\begin{align*}
&\EE[S_A(\eta_t+\sum_{i=1}^{k-1}\delta_{x_i}))-S_A(\eta_t)] \\
& = \sum_{\emptyset\neq I \subset \{1,\hdots,k-1\}} {k!\over (k-|I|)!}\, \EE \sum_{(y_1,\hdots,y_{k-|I|})\in\eta^{k-|I|}_{t,\neq}} h(x_I,y_1,\hdots,y_{k-|I|})\\
& = \sum_{\emptyset\neq I \subset \{1,\hdots,k-1\}} {k!\over (k-|I|)!}\,h_{|I|}(x_I)\,,
\end{align*}
where for a subset $I=\{i_1,\ldots,i_j\}\subset\{1,\ldots,k-1\}$ we use the shorthand notation $x_I$ for $(x_{i_1},\ldots,x_{i_j})$. Hence,
\begin{align*}
& \int_{\XX^k} \big|\EE\big[\Delta f_B(S_A(\eta_t+\sum_{i=1}^{k-1}\delta_{x_i}))-\Delta f_B(S_A(\eta_t))]\big| \, h(x_1,\hdots,x_k) \, \mu_t^{k}(\dint(x_1,\hdots,x_k))\\
& \leq 2\varepsilon_{1,A} \int_{\XX^k} \sum_{\emptyset\neq I\subset\{1,\hdots,k-1\}} h_{|I|}(x_I) {k!\over (k-|I|)!}\, h(x_1,\hdots,x_k) \, \mu_t^{k}(\dint(x_1,\hdots,x_k))\\
& \leq 2\varepsilon_{1,A}\,k!(2^{k-1}-1) \varrho_A\,.
\end{align*}
This implies that
\begin{equation}\label{eqn:PartI}
\begin{split}
& \big|\EE\big[s_A \Delta f_B(S_A) - \Delta f_B(S_A) \int_\XX D_xS_A (-D_xL^{-1}S_A)\, \mu_t(\dint x)\big]\big|\\
&\leq \varepsilon_{1,A}\bigg(\, k \sum_{\ell=1}^k |\Pi_{\neq}(k-1,\ell-1)|  +2k!( 2^{k-1}-1)\bigg)\,\varrho_A =:C_1\,\varepsilon_{1,A}\varrho_A\,.
\end{split}
\end{equation}
For the second term in \eqref{eqn:BoundPoissonApproximation} we have
\begin{align*}
& 2 \int_\XX \EE[ \max\{0,D_xS_A-1\}(-D_xL^{-1}S_A)] \, \mu_t(\dint x) \\
& \leq \frac{2}k \int_\XX \EE[ \max\{0,D_xS_A-1\}D_xS_A] \, \mu_t(\dint x)\\
& \quad + 2 \int_\XX \EE[ \max\{0,D_xS_A-1\}  \ |D_xL^{-1}S_A+D_xS_A/k|] \, \mu_t(\dint x) \allowdisplaybreaks\\
& \leq \frac{2}{k} \int_\XX \EE[(D_xS_A-1)D_xS_A] \, \mu_t(\dint x)\\
& \quad + 2 \int_\XX \EE[ \sqrt{D_xS_A (D_xS_A-1)} \ |D_xL^{-1}S_A+D_xS_A/k|] \, \mu_t(\dint x)\\
& \leq 3 \int_\XX \EE[(D_xS_A-1)D_xS_A] \, \mu_t(\dint x) + \int_\XX \EE[ |D_xL^{-1}S_A+D_xS_A/k|^2] \, \mu_t(\dint x)\,.
\end{align*}
It follows from Lemma \ref{lem:D} and Lemma \ref{lem:momentsUstatistics} that
\begin{align*}
& \int_\XX \EE[(D_xS_A-1)D_xS_A] \, \mu_t(\dint x)\\
 & = \int_\XX k^2\sum_{\sigma\in\Pi(k-1,k-1)} \int_{\XX^{|\sigma|}}(h(x,\cdot)\otimes h(x,\cdot))_\sigma \, \dint\mu_t^{|\sigma|}\, \mu_t(\dint x) -k \int_{\XX^k} h \, \dint\mu_t^k\,.
\end{align*}
Since there are $(k-1)!$ partitions with $|\sigma|=k-1$ and for each of them $$(h(x,\cdot)\otimes h(x,\cdot))_\sigma(x_1,\ldots,x_{|\sigma|})=\frac{1}{k!}h(x,x_1,\ldots,x_{|\sigma|})\,,$$ this leads to
\begin{align*}
& \int_\XX \EE[(D_xS_A-1)D_xS_A] \, \mu_t(\dint x)\\
 & =k^2 \hskip -0.2cm \sum_{\sigma\in\Pi_{\neq}(k-1,k-1)}  \int_\XX \int_{\XX^{|\sigma|}}(h(x,\cdot)\otimes h(x,\cdot))_\sigma \, \dint\mu_t^{|\sigma|}\, \mu_t(\dint x)\\
 & \leq k^2 |\Pi_{\neq}(k-1,k-1)| \varrho_A\,.
\end{align*}
Lemma \ref{lem:D} and Lemma \ref{lem:DL-1} imply that
$$
D_xL^{-1}S_A+D_xS_A/k=-\sum_{\ell=1}^{k-1} \sum_{(x_1,\hdots,x_{\ell-1})\in \eta^{\ell-1}_{t,\neq}} h_\ell(x,x_1,\hdots,x_{\ell-1})
$$
so that Lemma \ref{lem:momentsUstatistics} yields
\begin{align*}
& \int_\XX \EE[ |D_xL^{-1}S_A+D_xS_A/k|^2] \, \mu_t(\dint x)\\
& =\int_\XX\sum_{i,j=1}^{k-1} \sum_{\sigma\in\Pi(i-1,j-1)} \int_{\XX^{|\sigma|}} (h_i(x,\cdot)\otimes h_j(x,\cdot))_\sigma \, \dint\mu_t^{|\sigma|} \, \mu_t(\dint x)\\
& \leq \sum_{i,j=1}^{k-1} |\Pi(i-1,j-1)| \varrho_A\,.
\end{align*}
From the previous estimates, we can deduce that
\begin{equation}\label{eqn:PartII}
\begin{split}
& 2\varepsilon_{1,A} \int_\XX \EE[ \max\{0,D_xS_A-1\}(-D_xL^{-1}S_A)] \, \mu_t(\dint x) \\
& \leq \varepsilon_1\bigg( 3k^2 |\Pi_{\neq}(k-1,k-1)| +\sum_{i,j=1}^{k-1} |\Pi(i-1,j-1)| \bigg)  \varrho_A=:C_2\,\varepsilon_{1,A}\varrho_A\,.
\end{split}
\end{equation}
Combining \eqref{eqn:BoundPoissonApproximation} with \eqref{eqn:PartI} and \eqref{eqn:PartII} shows that
$$
d_{TV}(S_A,Y)\leq |s_A-s|+(C_1+C_2) \varepsilon_{1,A}\varrho_A\,,$$
which concludes the proof. \hfill $\Box$
\end{proof}

\begin{remark}
As already discussed in the introduction, the proof of Proposition \ref{prop:Poissonlimit} -- the main tool for the proof of Theorem \ref{thm:Main} -- is different from that given in \cite{STScalingLimits}. One of the differences is Lemma \ref{lem:DL-1}, which provides an explicit representation for $-D_xL^{-1}S_A$ based on Mehler's formula. We took considerable advantage of this in the proof of Proposition \ref{prop:Poissonlimit} and remark that the proof of the corresponding result in \cite{STScalingLimits} uses the chaotic decomposition of U-statistics and the product formula for multiple stochastic integrals (see \cite{LastChapter}). Another difference is that our proof here does not make use of the estimates established by the Malliavin-Chen-Stein method in \cite{Peccati2011}. Instead, we directly manipulate the Chen-Stein equation for Poisson approximation and this way improve the rate of convergence compared to \cite{STScalingLimits}. A different method to show Theorem \ref{thm:Main} and Theorem \ref{thm:MainBinomial} is the content of the recent paper \cite{DST}. 
\end{remark}

\subsection{Poisson approximation of classical U-statistics}\label{sec:PoissonApproxClassicalUStat}

In this section we consider U-statistics based on a binomial point process $\zeta_n$ defined as
$$S_A={1\over k!}\sum_{(x_1,\ldots,x_k)\in\zeta_{n,\neq}^k}\I\{f(x_1,\ldots,x_k)\in A\}\,,$$
where $f$ is as in Section \ref{sec:Results} and $A\subset\RR$ is bounded and measurable. Recall that in the context of a binomial point process $\zeta_n$ we assume that $\mu(\XX)=1$. Denote as in the previous section by $s_A:=\EE[S_A]$ the expectation of $S_A$. Notice that
\begin{equation}\label{eq:ExpectationUStatClassical}
s_A = (n)_k\int_{\XX^k}h(x_1,\ldots,x_k)\,\mu^k(\dint(x_1,\ldots,x_k))
\end{equation}
with $h(x_1\ldots,x_k)=(k!)^{-1}\I\{f(x_1,\ldots,x_k)\in A\}$.

\begin{proposition}\label{prop:PoissApproxClassicalUStat}
Let $S_A$ be as above and let $Y$ be a Poisson distributed random variable with mean $s>0$ and define
\begin{align*}
\varrho_A:= \max_{1\leq\ell\leq k-1} (n)_{2k-\ell}\int_{\XX^\ell}  \bigg(&\int_{\XX^{k-\ell}} h(x_1,\ldots\\
&\hdots,x_k) \, \mu^{k-\ell}(\dint(x_{\ell+1},\hdots,x_k)) \bigg)^2\mu^\ell(\dint(x_1,\hdots,x_\ell))\,.
\end{align*}
Then there is a constant $C\geq 1$ only depending on $k$ such that
$$d_{TV}(S_A,Y)\leq|s_A-s|+C\,\min\Big\{1,{1\over s_A}\Big\}\Big(\varrho_A+{s_A^2\over n}\Big)\,.$$
\end{proposition}
\begin{proof}
By the same arguments as at the beginning of the proof of Proposition \ref{prop:Poissonlimit} it is sufficient to assume that $s=s_A$ in what follows. To simplify the presentation we put $N:=\{I\subset\{1,\ldots,n\}:|I|=k\}$ and re-write $S_A$ as
$$S_A=\sum_{I\in N}\I\{f(X_I)\in A\}\,,$$ where $X_1,\ldots,X_n$ are i.i.d.\ random elements in $\XX$ with distribution $\mu$ and where $X_I$ is shorthand for $(X_{i_1},\ldots,X_{i_k})$ if $I=\{i_1,\ldots,i_k\}$. In this situation it follows from Theorem 2 in \cite{BarbourEagleson} that
\begin{align*}
d_{TV}(S,Y)\leq &\min\Big\{1,{1\over s_A}\Big\}\sum_{I\in N}\Big(\PP(f(X_I)\in A)^2+\sum_{r=1}^{k-1}\sum_{J\in N\atop |I\cap J|=r}\PP(f(X_I)\in A)\PP(f(X_J)\in A)\Big)\\
&+\min\Big\{1,{1\over s_A}\Big\}\sum_{I\in N}\sum_{r=1}^{k-1}\sum_{J\in N\atop |I\cap J|=r}\PP(f(X_I)\in A,f(X_J)\in A)\,.
\end{align*}
Since $s_A=\EE[S_A]={(n)_k\over k!}\PP(f(X_1,\ldots,X_k)\in A)$, we have that
\begin{align*}
&\sum_{I\in N}\Big(\PP(f(X_I)\in A)^2+\sum_{r=1}^{k-1}\sum_{J\in N\atop |I\cap J|=r}\PP(f(X_I)\in A)\PP(f(X_J)\in A)\Big)\\
&={(n)_k\over k!}\Big(\Big({k!\over (n)_k}s_A\Big)^2+\sum_{r=1}^{k-1}\sum_{J\in N\atop |I\cap J|=r}\Big({k!\over (n)_k}s_A\Big)^2\,\Big)\\
&={k!\over (n)_k}s_A^2\Big(1+\sum_{r=1}^{k-1}{k\choose r}{n-k\choose k-r}\,\Big)\\
&\leq {k!\over (n)_k}s_A^2\,2^k(n-1)_{k-1}\\
&\leq {2^kk!s_A^2\over n}\,.
\end{align*}
For the second term we find that
\begin{align*}
&\sum_{I\in N}\sum_{r=1}^{k-1}\sum_{J\in N\atop |I\cap J|=r}\PP(f(X_I)\in A,f(X_J)\in A)\\
&={(n)_k\over k!}\sum_{r=1}^{k-1}{k\choose r}{n-k\choose k-r}\PP(f(X_1,\ldots,X_k)\in A,f(X_1,\ldots,X_r,X_{k+1},\ldots,X_{2k-r})\in A)\\
&\leq {(n)_k\over k!}\sum_{r=1}^{k-1}{k\choose r}{n-k\choose k-r}{(k!)^2\over (n)_{2k-r}}\varrho_A\\
&\leq 2^kk!\,\varrho_A\,.
\end{align*}
Putting $C:=2^kk!$ proves the claim.\hfill $\Box$
\end{proof}

\subsection{Proofs of Theorem \ref{thm:Main} and \ref{thm:MainBinomial} and Corollary \ref{corol:Main} and \ref{corol:MainBinomial}}\label{subsec:ProofsMain}

\begin{proof}[of Theorem \ref{thm:Main}]
We define the set classes
$$
{\bf I}=\{I=(a,b]: a,b\in\RR, a<b\}
$$
and
$$
{\bf V}=\{V=\bigcup_{i=1}^n I_i: n\in\NN, I_i\in{\bf I}, i=1,\hdots,n\}.
$$
From \cite[Theorem 16.29]{Kallenberg} it follows that $(t^\gamma \xi_t)_{t>0}$ converges in distribution to a Poisson point process $\xi$ with intensity measure $\nu$ if
\begin{equation}\label{eqn:ConditionConvergenceI}
\lim\limits_{t\to\infty} \PP(\xi_t(t^{-\gamma}V)=0)=\PP(\xi(V)=0)=\exp(-\nu(V)), \quad V\in{\bf V}\,,
\end{equation}
and
\begin{equation}\label{eqn:ConditionConvergenceII}
\lim\limits_{t\to\infty} \PP(\xi_t(t^{-\gamma}I)>1)=\PP(\xi(I)>1)=1-(1+\nu(I))\exp(-\nu(I)), \quad I\in{\bf I}\,.
\end{equation}
Note that ${\bf I}\subset{\bf V}$ and that every set $V\in{\bf V}$ can be represented in the form
$$
V=\bigcup_{i=1}^n (a_i,b_i] \quad \text{with} \quad a_1<b_1<\hdots<a_n<b_n \quad \text{and} \quad n\in\NN\,.
$$
For $V\in {\bf V}$ we define the Poisson U-statistic
$$
S_{V,t}=\frac{1}{k!} \sum_{(x_1,\hdots,x_k)\in\eta_{t,\neq}^k} \I\{f(x_1,\hdots,x_k)\in t^{-\gamma} V\}\,,
$$
which has expectation
\begin{align*}
\EE[S_{V,t}] & = \frac{1}{k!} \EE \sum_{(x_1,\hdots,x_k)\in\eta_{t,\neq}^k} \I\{f(x_1,\hdots,x_k)\in t^{-\gamma} V\}\\
& = \sum_{i=1}^n \frac{1}{k!} \EE\sum_{(x_1,\hdots,x_k)\in\eta_{t,\neq}^k} \I\{f(x_1,\hdots,x_k)\in t^{-\gamma} (a_i,b_i]\}=\sum_{i=1}^n \alpha_t(a_i,b_i).
\end{align*}
Since $\xi(V)$ is Poisson distributed with mean $\nu(V)=\sum_{i=1}^n \nu((a_i,b_i])$, it follows from Proposition \ref{prop:Poissonlimit} that
$$
d_{TV}(S_{V,t},\xi(V))\leq \bigg|\sum_{i=1}^n \alpha_t(a_i,b_i) - \sum_{i=1}^n \nu((a_i,b_i]) \bigg|+C\,r_t(y_{max})
$$
with $y_{max}:=\max\{|a_1|,|b_n|\}$ and $C\geq 1$. Now, assumptions \eqref{eqn:MainAssumptionI} and \eqref{eqn:MainAssumptionII} yield that
$$
\lim\limits_{t\to\infty} d_{TV}(S_{V,t},\xi(V)) =0\,.
$$
Consequently, the conditions \eqref{eqn:ConditionConvergenceI} and \eqref{eqn:ConditionConvergenceII} are satisfied so that $(t^\gamma \xi_t)_{t>0}$ converges in distribution to $\xi$. Choosing $V=(0,y]$ and using the fact that $t^\gamma M_t^{(m)}>y$ is equivalent to $S_{(0,y],t}<m$ lead to the first inequality in Theorem \ref{thm:Main}. The second one follows analogously from $V=(-y,0]$ and by using the equivalence of $t^\gamma M_t^{(-m)}\geq y$ and $S_{(-y,0],t}<m$.\hfill $\Box$
\end{proof}

\begin{proof}[of Corollary \ref{corol:Main}]
Theorem \ref{thm:Main} with $\nu$ defined as in \eqref{eqn:nu} yields the assertions of Corollary \ref{corol:Main}.\hfill $\Box$
\end{proof}

\begin{proof}[of Theorem \ref{thm:MainBinomial} and Corollary \ref{corol:MainBinomial}]
Since the proofs are similar to those of Theorem \ref{thm:Main} and Corollary \ref{corol:Main}, we skip the details.\hfill $\Box$
\end{proof}

\end{document}